\documentclass[10pt,twoside,twocolumn]{IEEEtran}
\usepackage{amsmath,graphicx}
\graphicspath{{./img/}}
\pdfoutput=1

\usepackage{enumitem}

\usepackage{algorithm} 
\usepackage{algorithmic}

\usepackage{amssymb,subfigure} 
\usepackage{amsfonts}
\usepackage{amsthm}
\usepackage{booktabs}
\usepackage{url}
\usepackage{mathtools}
\usepackage{cite}

\newcommand{\eg}{{\it e.g.}}  
\newcommand{\ie}{{\it i.e.}}
\newcommand{\ea}{{\it et al.}}

\newcommand{\ones}{\mathbf 1} 
\newcommand{\reals}{{\mathbb R}}

\newcommand{\naturals}{\mathbb{N}}
\newcommand{\expect}{{\mathbb E}}
\newcommand{\prob}{{\mathbb P}}

\newcommand{\argmin}{\operatornamewithlimits{argmin}}

\newcommand{\minimize}{\operatornamewithlimits{minimize}}

\newcommand{\st}{\operatornamewithlimits{\text{subject to}}}

\newcommand {\mtxnorm}[1]{\left|\mathopen{}\left|\mathopen{}\left|
        {#1} \right|\mathclose{}\right|\mathclose{}\right|}

\newtheorem{theorem}{Theorem} 
\newtheorem{lemma}[theorem]{Lemma}
\newtheorem{proposition}[theorem]{Proposition}
\newtheorem{assumption}[theorem]{Assumption}
\newtheorem{definition}[theorem]{Definition}
\title{Dealing with bad apples: Robust range-based network
  localization via distributed relaxation methods}
\author{Cl\'audia Soares*,~\IEEEmembership{Member,~IEEE,} and Jo\~ao
  Gomes,~\IEEEmembership{Member,~IEEE} \thanks{This research was
    partially supported by Funda\c{c}\~{a}o para a Ci\^{e}ncia e a
    Tecnologia (project UID/EEA/50009/2013) and EU-H2020 project
    WiMUST (grant agreement No. 645141) The authors are with the
    Institute for Systems and Robotics (ISR), Instituto Superior
    T\'ecnico, Universidade de Lisboa, 1049-001 Lisboa, Portugal
    (e-mail: \{csoares,jpg\}@isr.tecnico.ulisboa.pt).}}
\begin{document}
%
\maketitle
\begin{abstract}
  Real-world network applications must cope with failing nodes,
  malicious attacks, or, somehow, nodes facing corrupted data ---
  classified as outliers. One enabling application is the geographic
  localization of the network nodes. However, despite excellent work
  on the network localization problem, prior research seldom
  considered outlier data --- even now, when already deployed networks
  cry out for robust procedures. We propose robust, fast, and
  distributed network localization algorithms, resilient to high-power
  noise, but also precise under regular Gaussian noise. We use the
  Huber M-estimator as a difference measure between the distance of
  estimated nodes and noisy range measurements, thus obtaining a
  robust (but nonconvex) optimization problem. We then devise a convex
  underestimator solvable in polynomial time, and tight in the
  inter-node terms. We also provide an optimality bound for the convex
  underestimator. We put forward a new representation of the Huber
  function composed with a norm, enabling distributed robust
  localization algorithms to minimize the proposed underestimator. The
  synchronous distributed method has optimal convergence rate and the
  asynchronous one converges in finite time, for a given
  precision. The main highlight of our contribution lies on the fact
  that we pay no price for distributed computation nor in accuracy,
  nor in communication cost or convergence speed. Simulations show the
  advantage of using our proposed algorithms, both in the presence of
  outliers and under regular Gaussian noise: our method exceeds the
  accuracy of an alternative robust approach based on L1 norms by at
  least 100m in an area of 1Km sides.
\end{abstract}

\begin{IEEEkeywords}
  Distributed algorithms, Robust estimation, Huber function, convex
  relaxations, nonconvex optimization, maximum likelihood estimation,
  distributed iterative network localization, sensor networks.
\end{IEEEkeywords}

\begin{center} \bfseries EDICS Category: OPT-CVXR OPT-DOPT NEG-APPL
  NEG-LOCL \end{center}
\IEEEpeerreviewmaketitle

\section{Introduction}
\label{sec:intro}
Outliers can cause large  errors in non robust estimation algorithms,
and, if other systems use wrong estimates as input, error propagation
can invalidate the the engineered system's final purpose. Network
localization is a key component in many network-centric systems that
is prone to such catastrophic error propagation. 
It might be taken for granted in most sensor network applications, but
in challenging environments network localization is an open and very
active research field. We present a new approach addressing the
presence of outliers, not by eliminating them from the estimation
process, but by weighting them, so they can contribute to the
solution, while mitigating the outlier bias on the estimator.

\subsection{The problem}
\label{sec:problem}

The network is represented as an undirected graph~$\mathcal{G} =
(\mathcal{V},\mathcal{E})$. We represent the set of sensors with
unknown positions as~$\mathcal{V} = \{1,2, \dots, n\}$. There is an
edge $i \sim j \in {\mathcal E}$ between nodes $i$ and $j$ if a range
measurement between~$i$ and $j$ is available and~$i$ and $j$ can
communicate with each other.  Anchors have known positions and are
collected in the set ${\mathcal A} = \{ 1, \ldots, m \}$; they are not
nodes on the graph~$\mathcal{G}$. For each sensor $i \in {\mathcal
  V}$, we let ${\mathcal A}_i \subset {\mathcal A}$ be the subset of
anchors with measured range to node~$i$. The set~$N_{i}$ collects the
neighbor sensor nodes of node~$i$.

The element positions belong to~$\reals^p$ with~$p=2$ for planar
networks, and $p=3$ for volumetric ones.  We denote by $x_i \in \reals^p$ the
position of sensor $i$, and by $d_{ij}$ the range measurement between
sensors $i$ and $j$. Anchor positions
are denoted by $a_{k} \in \reals^{p}$. We let $r_{ik}$ denote the
noisy range measurement between sensor $i$ and anchor $k$.

We aim at estimating the sensor positions~$x=\{x_{\mathcal{V}}\}$,
taking into account two types of noise: (1) regular Gaussian noise,
and (2) outlier induced noise.

\subsection{Related work}
\label{sec:related-work}

Focusing on recent work, several different approaches are available,
some performing semi-definite or weaker second-order cone relaxations
of the original nonconvex problem like O\u{g}uz-Ekim
\ea~\cite{OguzGomesXavierOliveira2011} or Biswas
\ea~\cite{BiswasLiangTohYeWang2006}. These approaches do not scale
well, since the centralized semidefinite program (SDP) or second-order
cone program (SOCP) gets very large even for a small number of nodes.
In~O\u{g}uz-Ekim \ea\ the Majorization-Minimization (MM) framework was
used with quadratic cost functions to also derive centralized
approaches to the sensor network localization problem.  Other
approaches rely on multidimensional scaling, where the sensor network
localization problem is posed as a least-squares problem, as in Shang
\ea~\cite{ShangRumiZhangFromherz2004}. Unfortunately, multidimensional
scaling is unreliable in large-scale networks with sparse
connectivity. Also relying on the well-tested weighted least-squares
approach, the work of Destino and Abreu~\cite{DestinoAbreu2011}
performs successive minimizations of a weighted least-squares cost
function convolved with a Gaussian kernel of decreasing variance,
following an homotopy scheme.  Another class of relaxations are convex
envelopes of terms in the cost function, like Soares
\ea~\cite{soares2014simple}.

Some previous references directly tackle the nonconvex maximum
likelihood problem, aspiring to no more than a local minimizer, whose
goodness depends on the quality of the initialization. Here we can
point out several approaches, like Costa \textit{et
  al.}~\cite{CostaPatwariHero2006}, where the authors present a
distributed refinement solution inspired in multidimensional scaling,
or Calafiore \textit{et al.}~\cite{CalafioreCarloneWei2010},
presenting a gradient algorithm with Barzilai-Borwein step sizes
calculated in a first consensus phase at every algorithm step, and
Soares \textit{et al.}~\cite{SoaresXavierGomes2014a}, where the
authors reformulate the problem to obtain a Lipschitz gradient cost;
shifting to this cost function enables a MM approach based on quadratic upper bounds that decouple across
nodes; the resulting algorithm is distributed, with all nodes working
in parallel.

All these approaches assume Gaussian noise contaminating the distance
measurements or their squares, while many empirical studies reveal
that real data are seldom Gaussian. Despite this, the literature
is scarce in robust estimation techniques for network
localization. Some of the few approaches rely on identifying outliers
from regular data and discarding them. An example is Ihler \textit{et
  al.}~\cite{IhlerFisherMosesWillsky2005}, which formulates network
localization as an inference problem in a graphical model. To
approximate an outlier process the authors add a high-variance
Gaussian to the Gaussian mixtures and employ nonparametric belief
propagation to approximate the solution. The authors assume a
particular probability distribution for outlier measurements. In the
same vein, Ash \textit{et al.}~\cite{AshMoses2005} employs the EM
algorithm to jointly estimate outliers and sensor positions. Recently,
the work of Yin \ea ~\cite{YinZoubirFritscheGustafsson2013} tackled
robust localization with estimation of positions, mixture parameters,
and outlier noise model for unknown propagation conditions, again
under predetermined probability distributions. By removing guessed
outliers from the estimation process some information is lost.

Alternatively, methods may perform a soft rejection of outliers, still
allowing them to contribute to the solution.  O\u{g}uz-Ekim
\ea~\cite{OguzGomesXavierOliveira2011} derived a maximum likelihood
estimator for Laplacian noise and relaxed it to a convex program by
linearizing and dropping a rank constraint; they also proposed a
centralized algorithm to solve the approximated problem. Such
centralized solutions fail to scale with the number of nodes and
number of collected measurements. Forero and
Giannakis~\cite{ForeroGiannakis2012} presented a robust
multidimensional scaling based on regularized least-squares, where the
regularization term was surrogated by a convex function, and solved
via MM. The main drawbacks of this approach are the centralized
processing architecture and selection of a sensitive regularization
parameter. Korkmaz and van der Veen~\cite{KorkmazVeen2009} use the
Huber loss~\cite{huber1964} composed with a discrepancy between
measurements and estimated distances, in order to achieve robustness to
outliers. The resulting cost is nonconvex, and optimized by means of
the Majorization-Minimization technique. The method is distributed,
but the quality of the solution depends in the quality of the
initialization.
Yousefi \ea~\cite{YousefiChangChampagne2014} extends the Projection
Onto Convex Sets approach in Blatt and Hero~\cite{BlattHero2006} to
the Huber loss. The approximation is then solved via a coordinate
descent algorithm, where a ``one-at-a-time'' (sequential) update
scheme is critical for convergence; so this solution depends on 
knowledge of a Hamiltonian path in the network --- a known NP-complete
problem.

\subsection{Contributions}
\label{sec:contributions}

In applications of large-scale networks there is a need for a
distributed localization method for soft rejection of outliers that is
simple, scalable and efficient under any outlier noise
distribution. The method we present incorporates outliers into the
estimation process and does not assume any statistical outlier model. We
capitalize on the robust estimation properties of the Huber function
but, unlike Korkmaz and van der Veen, we do not address the nonconvex
cost in our proposal, thus removing the initialization
uncertainty. Instead, we derive a convex relaxation which numerically
outperforms state-of-the-art methods, and other natural formulations of
the problem. The contributions of this work are:

\begin{enumerate}
\item We motivate a tight convex underestimator for each term of the
  robust discrepancy measure for sensor network localization
  (Section~\ref{sec:conv-under});
\item We provide an optimality bound for the convex relaxation, and we
  analyze the tightness of the convex approximation. We also compare
  it with other discrepancy measures and appropriate relaxations. All
  measurements contribute to the estimate, although we do not compute
  specific weights. Numerical simulations illustrate the quality of
  the convex underestimator
  (Section~\ref{sec:appr-qual-cvx-underestimator});
\item We put forth a new representation of the Huber function
  composed with a norm (Section~\ref{sec:repr-huber-comp});
\item Capitalizing on the previous contributions, we develop a
  gradient method which is distributed, requires only simple
  computations at each node, and has guaranteed optimal convergence rate
  (Sections~\ref{sec:synchr-algor},
  and~\ref{sec:analysis-synchr-algor});
\item Further, we introduce an asynchronous method for robust network
  localization, with convergence guarantees
  (Sections~\ref{sec:asynchr-algor},
  and~\ref{sec:analysis-asynchr-algor});
\item We benchmark our algorithms with the state-of-the-art in robust
  network localization, achieving better performance with fewer
  communications (Section~\ref{sec:numer-exper}).
\end{enumerate}
Both solutions proposed in this paper do not assume knowledge of a
Hamiltonian path in the network. Also, the proposed scheme has the
fastest possible convergence for a first order method, in the
synchronous case, without degradation of accuracy. This is possible
after analyzing the novel representation of the robust problem, as
described in the sequel, and uncovering that the problem is, in fact,
naturally distributed.

\section{Discrepancy measure}
\label{sec:discrepancy-measure}

The maximum-likelihood estimator for  sensor positions with
additive i.i.d.\ Gaussian noise contaminating range measurements is the
solution of the optimization problem
\begin{equation*}
  \label{eq:snlOptimizationProblem} 
  \minimize_{x} g_{G}(x),
\end{equation*} 
where
\begin{align*} 
\label{eq:sqdist}
  g_{G}(x) = & \sum _{i \sim j} \frac{1}{2}(\|x_{i} - x_{j}\| -
  d_{ij})^2 + \\
 &\sum_{i} \sum_{k \in \mathcal{A}_{i}} \frac{1}{2}(\|x_{i}-a_{k}\| -
 r_{ik})^2.
\end{align*} 
However, outlier measurements are non-Gaussian and will heavily bias
the solutions of the optimization problem since their magnitude will
be amplified by the squares~$h_{Q}(t) = t^{2}$ in each outlier term.
Robust estimation theory provides some alternatives to perform soft
rejection of outliers, namely, using the~$L_{1}$
loss~$h_{|\cdot|}(t) = |t|$ or the Huber loss
\begin{equation}
\label{eq:huber-loss}
h_{R}(t) =
\begin{cases}
  t^{2} & \text{if } |t| \leq R,\\
  2R|t|-R^{2} & \text{if } |t| \geq R.
\end{cases}
\end{equation}
The Huber loss achieves the best of two worlds: it is robust for large values
of the argument --- like the~$L_{1}$ loss --- and for reasonable noise
levels it behaves like~$g_{Q}$, thus leading to the maximum-likelihood
estimator adapted to regular noise.
\begin{figure}[tb]
  \centering
  \includegraphics[width=\columnwidth]{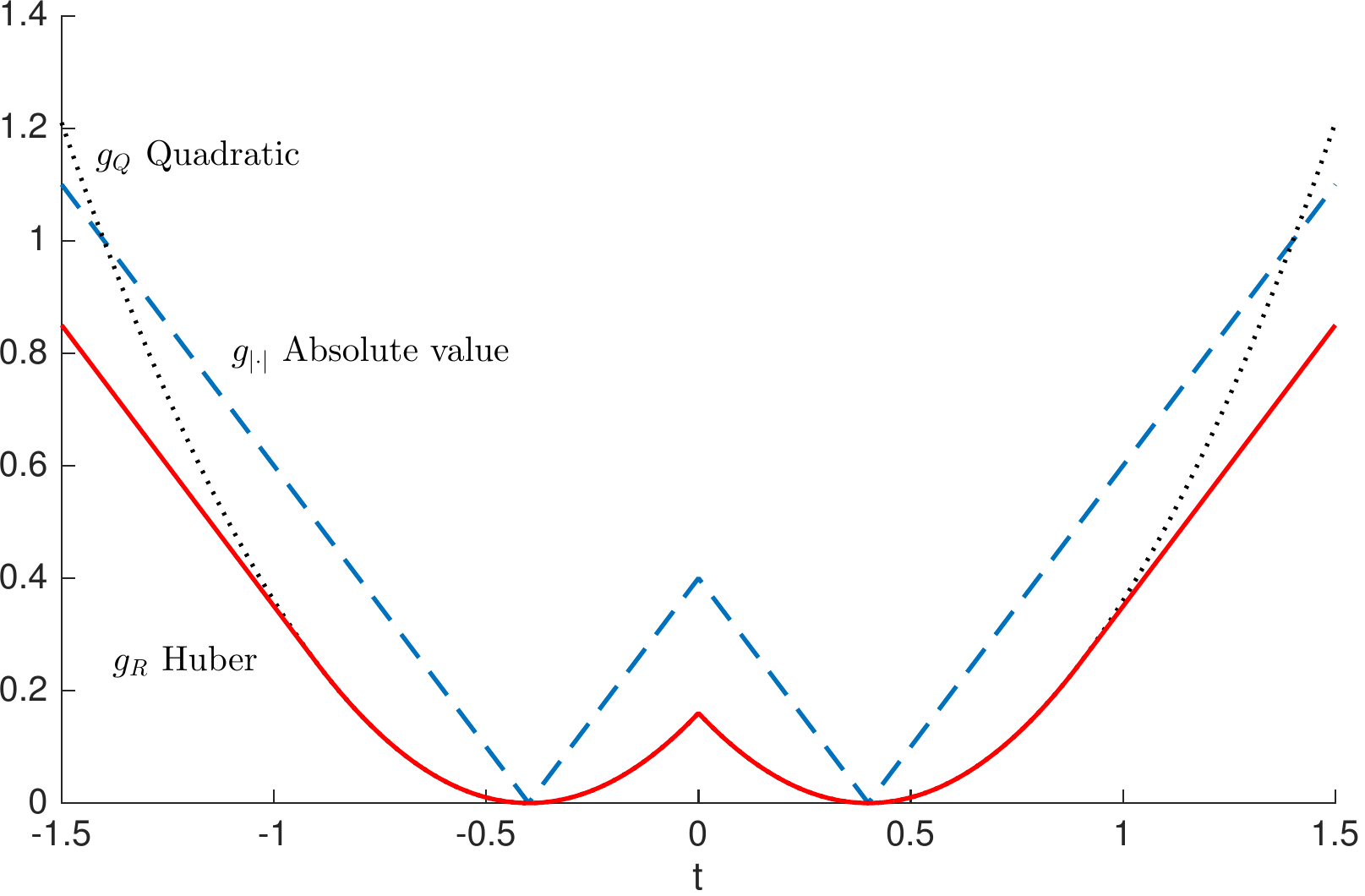}
  \caption{The different cost functions considered in this paper,
    applied to a 1D, one-edge problem, where an anchor sits at the
    origin, and the sensor at 0.4. The maximum-likelihood independent
    white Gaussian noise is~$g_{Q}(t) = (|t| - d)^{2}$ and shows the
    steepest tails, which act as outlier amplifiers; the $L_{1}$
    loss~$g_{|\cdot|}(t) = ||t| - d|$, associated with impulsive
    noise, fails to model the Gaussianity of regular operating noise;
    and, finally, the Huber loss~$g_{R}(t) = h_{R}(|t| - d)$, combines
    robustness to high-power outliers and adaptation to medium-power
    Gaussian noise.}
  \label{fig:nonconvex}
\end{figure}
Figure~\ref{fig:nonconvex} depicts a one-dimensional example of these different
costs. We can observe in this simple example the main properties of
the different cost functions, in terms of adaptation to low/medium-power
Gaussian noise and high-power outlier spikes.
Using~\eqref{eq:huber-loss} we can write our robust optimization problem as
\begin{equation}
  \label{eq:snlOptProb}
  \minimize_{x} g_{R}(x)
\end{equation}
where
\begin{align}
  \nonumber
    g_{R}(x) = & \sum _{i \sim j} \frac 12 h_{R_{ij}}(\|x_{i} - x_{j}\| - d_{ij})
    \; + \\ \label{eq:huberDist}
    &\sum_{i} \sum_{k \in \mathcal{A}_{i}} \frac 12 h_{R_{ik}}(\|x_{i}-a_{k}\| - r_{ik}).
\end{align}
This function is nonconvex and, in general, difficult to minimize.
We shall provide a convex underestimator that tightly bounds each term
of~\eqref{eq:huberDist}, thus leading to better estimation results
than other relaxations which are not tight~\cite{SimonettoLeus2014}.

\section{Convex underestimator}
\label{sec:conv-under}
To convexify~$g_{R}$ we can replace each term by its
convex hull\footnote{The convex
  hull of a function~$\gamma$, \ie, its best
  possible convex underestimator, is defined as $ \text{conv } \gamma(x) =
  \sup\left \{ \eta(x) \; : \; \eta \leq \gamma, \; \eta \text{ is
      convex} \right \} $. It is hard to determine in general~\cite{UrrutyMarechal1993}.},
\begin{figure}[tb]
  \centering
  \includegraphics[width=\columnwidth]{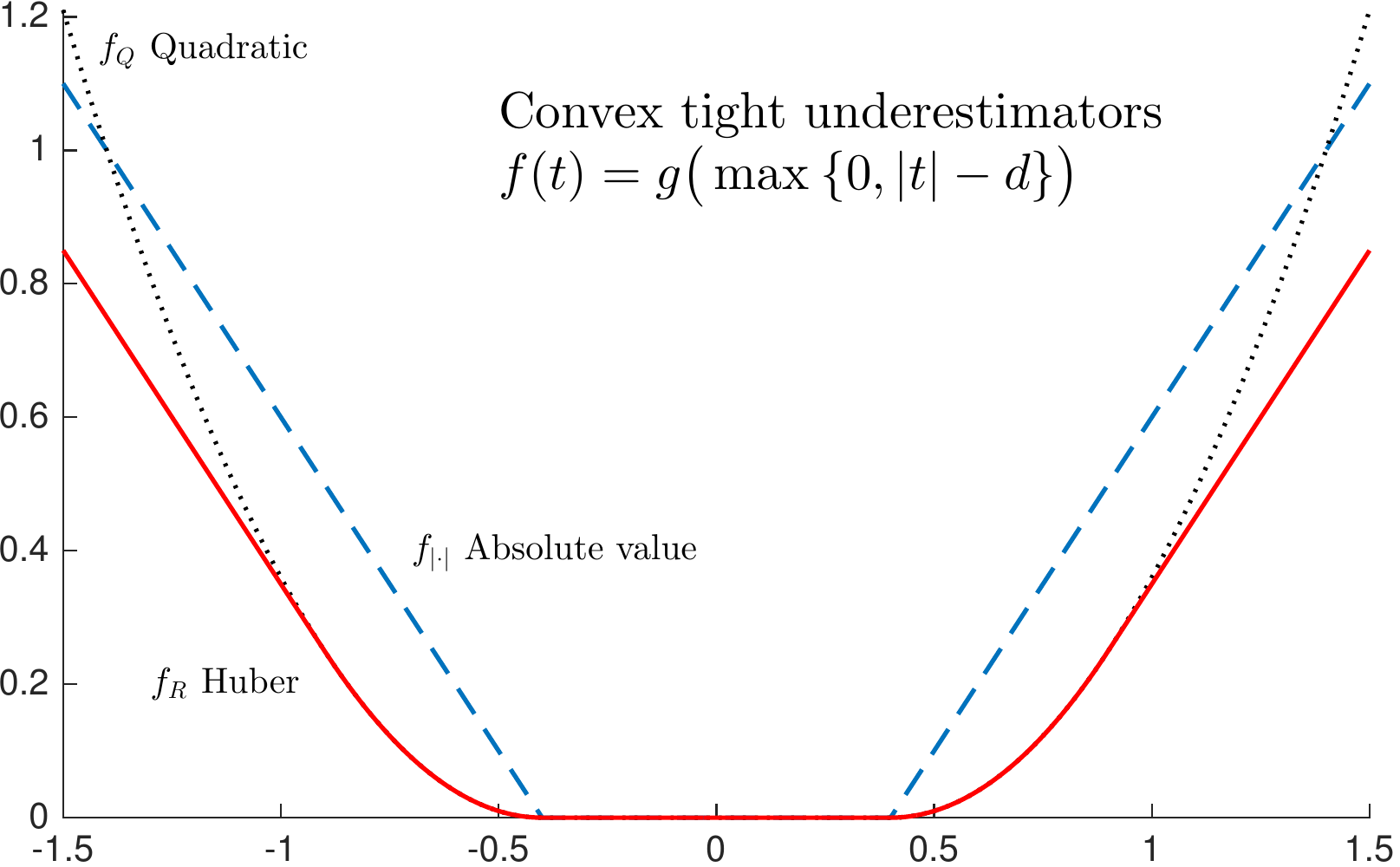}
  \caption{All functions~$f$ are tight underestimators of the
    functions~$g$ in Figure~\ref{fig:nonconvex}. They are the convex
    envelopes and, thus, the best convex approximations to each one of
    the original nonconvex cost terms. The convexification is
    performed by restricting the arguments of~$g$ to be nonnegative.}
  \label{fig:convexified}
\end{figure}
as depicted in Figure~\ref{fig:convexified}. Here, we observe that the
high-power behavior is maintained, whereas the medium/low-power is
only altered in the convexified area. We define the convex costs by
composing any of the convex functions~$h$ with a nondecreasing
function~$s$
 \begin{equation*}
   s(t) = \max \{ 0, t\}
 \end{equation*}
which, in turn, operates on the discrepancies
\begin{align*}
  \delta_{ij}(x) = \|x_{i} - x_{j}\| - d_{ik}, \\
  \delta_{ik}(x_{i}) = \|x_{i} - a_{k}\| - r_{ik}.
\end{align*}
As~$s\left (\delta_{ij}(x)\right )$ and~$s\left (\delta_{ik}(x)\right
)$ are nondecreasing and each one of the functions~$h$ is convex,
then
\begin{equation}\label{eq:huber-cvx}
\begin{aligned}
  f_{R}(x) = &\sum _{i \sim j} \frac 12 h_{R_{ij}}\left(s\left(\|x_{i} - x_{j}\| -
      d_{ij}\right) \right) + \\
  &\sum_{i} \sum_{k \in \mathcal{A}_{i}} \frac 12 h_{Ra_{ik}}\left(s\left(\|x_{i}-a_{k}\|
      - r_{ik}\right) \right)
\end{aligned}
\end{equation}
is also convex. 
The cost function~\eqref{eq:huber-cvx} also appears in Yousefi
\ea~\cite{YousefiChangChampagne2014} via a distinct reasoning. But the
striking difference with respect to Yousefi \ea~is how the
cost~\eqref{eq:huber-cvx} is exploited here to generate distributed
solution methods where all nodes work in parallel, for the synchronous
algorithm, or are randomly awaken, for the asynchronous
algorithm.

\subsection{Approximation quality of the  convex underestimator}
\label{sec:appr-qual-cvx-underestimator}

The quality of the convexified quadratic problem was
addressed in~\cite{soares2014simple}, which we summarize here for
the reader's convenience and extend to the two other convex problems.

The optimal value of the nonconvex~$g$, denoted by~$g^{\star}$, is
bounded by
\begin{equation*}
  f^{\star} = f(x^{\star}) \leq g^{\star} \leq g(x^{\star}),
\end{equation*}
where~$x^{\star}$ is the minimizer of the convex underestimator~$f$,
and
\begin{equation*}
  f^{\star} = \min_{x} f(x),
\end{equation*}
is the minimum of function~$f$. A bound for the optimality gap is,
thus,
\begin{IEEEeqnarray*}{rCl}
  g^{\star} - f^{\star} &\leq & g(x^{\star}) - f^{\star}.
\end{IEEEeqnarray*}
It is evident that in all cases (quadratic, Huber, and absolute
value)~$f$ is equal to~$g$ when~$\|x_{i} - x_{j}\| \geq d_{ij}$
and~$\|x_{i}-a_{k}\| \geq r_{ik}$. When the function terms differ,
say, for all edges\footnote{The same reasoning would apply to anchor
  terms.}~$i \sim j \in \mathcal{E}_{2} \subset \mathcal{E}$, we
have~$s\left(\|x_{i} - x_{j}\| - d_{ij}\right) = 0$, leading to
\begin{IEEEeqnarray}{rCl} \label{eq:bound-tight2} 
  g_{Q}^{\star} - f_{Q}^{\star} &\leq & \sum_{i \sim j \in \mathcal{E}_{2}} \frac 12
  \left(\|x_{i}^{\star} - x_{j}^{\star}\| - d_{ij}
  \right)^{2}\\ \label{eq:bound-tight1} 
  g_{| \cdot |}^{\star} - f_{| \cdot |}^{\star} &\leq & \sum_{i \sim j
    \in \mathcal{E}_{2}} \frac 12 \left| \|x_{i}^{\star} - x_{j}^{\star}\| - d_{ij}
  \right|\\ \label{eq:bound-tighth}
  g_{R}^{\star} - f_{R}^{\star} &\leq & \sum_{i \sim j \in
    \mathcal{E}_{2}} \frac 12 h_{R_{ij}}\left(\|x_{i}^{\star} -
      x_{j}^{\star}\| - d_{ij} \right),
\end{IEEEeqnarray}
where
\begin{equation*}
  \mathcal{E}_{2} = \{i \sim j \in \mathcal{E} :
\|x_i^\star - x_j^\star\| < d_{ij}) \}.
\end{equation*}
These bounds are an optimality gap guarantee available after the
convexified problem is solved; they tell us how low our estimates can
bring the original cost. Our bounds are tighter than the ones
available \textit{a priori} from applying~\cite[Th. 1]{udellBoyd2014},
which are
\begin{IEEEeqnarray}{rCl}\label{eq:udellBoyd2}
  g_{Q}^{\star} - f_{Q}^{\star} &\leq & \sum_{i \sim j} \frac 12
  d_{ij}^{2}\\ \label{eq:udellBoyd1}
  g_{| \cdot |}^{\star} - f_{| \cdot |}^{\star} &\leq & \sum_{i \sim j}
  \frac 12 d_{ij}\\ \label{eq:udellBoydh}
  g_{R}^{\star} - f_{R}^{\star} &\leq & \sum_{i \sim j} \frac 12
  h_{R_{ij}}\left( d_{ij} \right). 
\end{IEEEeqnarray}
For a single-node 1D example whose costs for a single noise
realization are exemplified in Figure~\ref{fig:nonconvexities},
\begin{figure}[tb]
  \centering 
\subfigure[Quadratic.]
  {\includegraphics[width=\columnwidth]{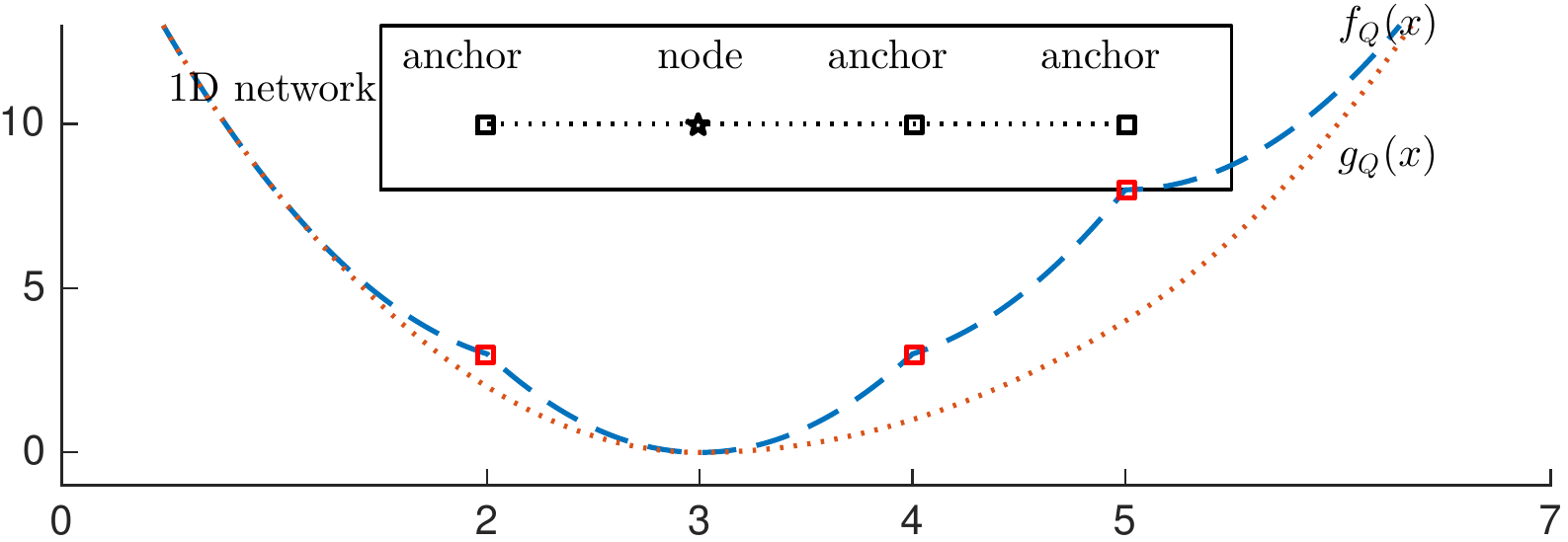}
  \label{fig:nonconvexities2}}
\subfigure[Absolute value.]
  {\includegraphics[width=\columnwidth]{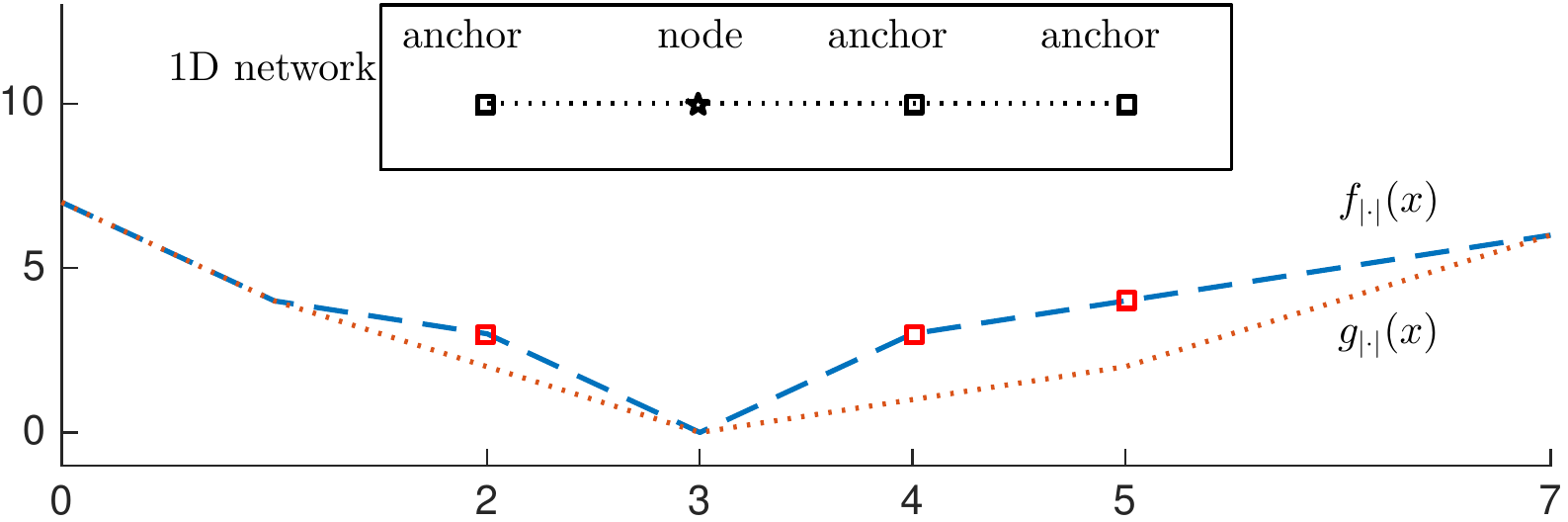}
  \label{fig:nonconvexities1}}
\subfigure[Robust Huber.]
  {\includegraphics[width=\columnwidth]{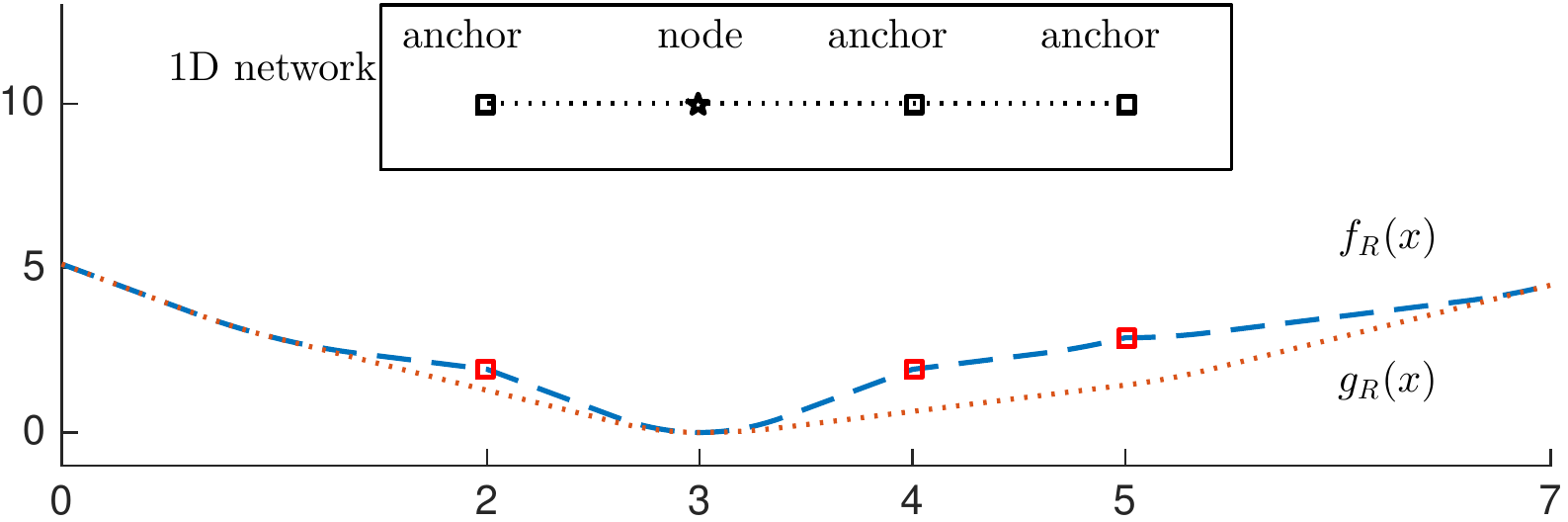}
  \label{fig:nonconvexitiesh}}
\caption{Single node 1D example of the quality of the approximation of
  the true nonconvex costs $g(x)$ by the convexified functions
  $f(x)$. The node positioned at $x=3$ has 3 neighboring anchors. The
  cost value is indicated in the vertical axis, while the tentative
  node position runs on the horizontal. The actual network is depicted
  above the plots.}
    \label{fig:nonconvexities}
\end{figure}
the bounds in~\eqref{eq:bound-tight2}-\eqref{eq:bound-tighth}
and~\eqref{eq:udellBoyd2}-\eqref{eq:udellBoydh},
\begin{table}[t]
  \caption{Bounds on the optimality gap for the example in Figure~\ref{fig:nonconvexities}}
  \label{tab:bounds}
  \centering
  \begin{tabular}[h]{@{}lccc@{}}
    \toprule
    Cost &
    \textbf{$g^{\star} - f^{\star}$} &
    \textbf{Eqs.~\eqref{eq:bound-tight2}-\eqref{eq:bound-tighth}} &
    \textbf{Eqs.~\eqref{eq:udellBoyd2}-\eqref{eq:udellBoydh}}\\\midrule
    Quadratic &3.7019&5.5250&11.3405\\
    Absolute value &1.1416&1.1533&3.0511\\
    Robust Huber &0.1784&0.1822&0.4786\\
    \bottomrule
  \end{tabular}
\end{table}
averaged over 500 Monte Carlo trials, are presented in
Table~\ref{tab:bounds}. The true average gap $g^{\star} - f^{\star}$
is also listed. In the Monte Carlo trials we sampled a set of zero
mean Gaussian random variables with $\sigma = 0.04$ for the baseline
Gaussian noise and obtained a noisy range measurement as
in~\eqref{eq:noise} below. One of the measurements is then corrupted
by a zero mean random variable with~$\sigma = 4$, modelling outlier
noise. These results show the tightness of the convexified function
under such noisy conditions and also demonstrate the looseness of the
\textit{a priori} bounds
in~\eqref{eq:udellBoyd2}-\eqref{eq:udellBoydh}. We can observe in
Figure~\ref{fig:nonconvexities} why the Huber-based relaxation will
perform better than the other two: not only do we use a robust
dissimilarity, but we also add a smaller optimality gap with the
surrogate. In the end, the Huber-based approximation will be tighter,
thus conferring robustness to the estimator, as pointed out by Destino
and Abreu~\cite{DestinoAbreu2011}.

\section{Distributed and robust sensor network localization}
\label{sec:distr-robust-sens}

We construct our algorithm by rewriting~\eqref{eq:huber-cvx} as the
infimum of a sum of Huber functions composed with a norm, and then by
rewriting each of the terms with an alternative representation that
uncovers the possibility of a naturally distributed, optimal method
for the estimation of the unknown sensor positions. For the first
step, we state the following:
\begin{proposition}
  \label{prop:cvx-term-norm}
  Each term of the first summation of~\eqref{eq:huber-cvx},
  corresponding to the edge~$i \sim j$, has a variational
  representation
  \begin{equation}
    \label{eq:cvx-term-norm}
    h_{R_{ij}}(s(\|x_{i}-x_{j}\| - d_{ij})) =\inf_{\|y_{ij}\|\leq d_{ij}} h_{R_{ij}}(\|x_{i}-x_{j}-y_{ij}\|),
  \end{equation}
  where~$y_{ij} \in \reals^{p}$ is an optimization variable.
\end{proposition}
The proof is detailed in Appendix~\ref{sec:proof-prop-cvx-term-norm}.

\subsection{Alternative representation of the Huber function composed
  with a norm}
\label{sec:repr-huber-comp}

Using the variational representation from
Proposition~\ref{prop:cvx-term-norm}, we can rewrite the convex
unconstrained minimization Problem~\eqref{eq:huber-cvx} as the
constrained problem
\begin{equation}
  \label{eq:cvx-aux-variables}
  \begin{aligned}
    \minimize_{x,y,w} &\sum_{i \sim j}\frac 12 h_{R_{ij}}(\|x_{i}-x_{j}-y_{ij}\|) + \\
    &\sum_{i} \sum_{k \in \mathcal{A}_{i}}
    \frac 12 h_{R_{ik}}(\|x_{i}-a_{k}-w_{ik}\|)\\
    \st &\left\{\|y_{ij}\| \leq d_{ij}, i \sim j\right\}\\
    &\left\{\|w_{ik}\| \leq r_{ik}, i \in \mathcal{V}, k \in
      \mathcal{A}_{i}\right\} 
  \end{aligned}
\end{equation}
where~$x = \{x_{i} : i \in \mathcal{V}\}$, $y = \{y_{ij} : i \sim j\}$,
and~$w = \{w_{ik} : i \in \mathcal{V}, k \in \mathcal{A}_{i}\}$. We
put forward a new representation of the Huber function
in~\eqref{eq:huber-loss}, when composed with the norm of a vector as
\begin{equation}
  \label{eq:huber-norm}
  \psi_{R}(u) = \|u\|^{2} - \mathrm{d}^{2}_{R}(u),
\end{equation}
where we denote by~$\mathrm{d}^{2}_{R}(u)$ the squared distance of
vector~$u$ to a ball of radius~$R$ centered at the origin. We use this
representation to rewrite the cost in
problem~\eqref{eq:cvx-aux-variables} as
\begin{equation*}
  \begin{aligned}
    \sum_{i \sim j} \frac 12 \|x_{i}-x_{j}-y_{ij}\|^{2} - \frac 12
    \mathrm{d}^{2}_{R_{ij}}(x_{i}-x_{j}-y_{ij}) + \\
    \sum_{i} \sum_{k \in \mathcal{A}_{i}}\frac 12
    \|x_{i}-a_{k}-w_{ik}\|^{2} - \frac 12
    \mathrm{d}^{2}_{R_{ik}}(x_{i}-a_{k}-w_{ik}),
  \end{aligned}
\end{equation*}
and we further work the problem, leading to
\begin{equation*}
  \begin{aligned}
    &\frac 12 \|Ax-y\|^{2} - \frac 12 \mathrm{d}^{2}_{\tilde R}(Ax-y) + \\
    &\sum_{i \in \mathcal{V}} \frac 12 \|x_{i}\otimes \ones -\alpha_{i}
    -w_{i}\|^{2} - \frac 12 \mathrm{d}^{2}_{\tilde Ra}(x_{i}\otimes
    \ones -\alpha_{i} -w_{i}),
  \end{aligned}
\end{equation*}
where~$\tilde{R}$ is the Cartesian product of the
balls~$\{x : \|x\| \leq R_{ij}, i \sim j\}$, $\tilde{Ra}$ is the
Cartesian product of the
balls~$\{x : \|x\| \leq R_{ik}, k \in \mathcal{A}_{i}, i \in
\mathcal{V}\}$,
and $\mathrm{d}^{2}_{\tilde R}(\cdot)$ is the squared distance to
set~$\tilde R$ (similarly for~$\mathrm{d}^{2}_{\tilde R a}(\cdot)$).
Matrix $A = C \otimes I$ is the Kronecker product of the arc-node
incidence matrix\footnote{The arc-node incidence matrix $C$ is a
  $|\mathcal{E}| \times |\mathcal{V}|$ matrix. The rows and the
  columns of $C$ are indexed by $\mathcal{E}$ and $\mathcal{V}$,
  respectively. The $(e,i)$-entry of $C$ is 0 if node $i$ and edge $e$
  are not incident, and otherwise it is 1 or -1 according to the
  direction agreed on the operation onset by the two nodes involved in
  edge~$e$.}~$C$ associated with the graph~$\mathcal{G}$, and the
identity matrix with dimension of the ambient space (usually, 2 or 3).
The terms~$\alpha_{i} = \{ a_{k} : k \in \mathcal{A}_{i}\}$ are the
collections of the positions of the anchors within range of each node,
and~$w_{i} = \{ w_{ik}, k \in \mathcal{A}_{i}\}$ are the collections
of associated variables.  Consider the aggregation of
variables~$z = (x,y,w)$; we define the constraint set
in~\eqref{eq:cvx-aux-variables} as
\begin{equation}
  \label{eq:constraint-set}
  \mathcal{Z} = \{(x,y,w) : \|y_{ij}\| \leq d_{ij}, i \sim j,
  \|w_{ik}\| \leq r_{ik}, k \in \mathcal{A}_{i}, i \in \mathcal{V}\},
\end{equation}
and the cost as
\begin{equation}
  \label{eq:matrix-cost}
  F(z) := \frac 12 \|Bz\|^{2} - \frac 12 \mathrm{d}^{2}_{\tilde{R}}(Bz) +
  \frac 12 \|Ez-\alpha\|^{2} - \frac 12 \mathrm{d}^{2}_{\tilde{R}a}(Ez-\alpha),
\end{equation}
where~$B = [A\quad -I \quad 0]$,
matrix~$E = [E_{x} \quad 0 \quad -I]$, and~$E_{x}$ is a selector
matrix of the anchor terms associated with each separate node.
With this notation,~\eqref{eq:cvx-aux-variables} becomes
\begin{equation}
  \label{eq:cvx-matrix-z}
  \begin{aligned}
    \minimize & \quad F(z)\\
    \st & \quad z \in \mathcal{Z}.
  \end{aligned}
\end{equation}

\subsection{Gradient}
\label{sec:gradient}
To compute the gradient of the cost in~\eqref{eq:matrix-cost}, we need
the gradient of the squared distance to a convex set --- a result from
convex analysis (see~\cite[Prop. X.3.2.2,
Th. X.3.2.3]{UrrutyMarechal1993}).
Let us denote the squared distance to the convex set~$C$ as
\begin{equation*}
  \phi(u) = \frac 12 \mathrm{d}^{2}_{C}(u).
\end{equation*}
 Then, from convex analysis, we know
that~$\phi$ is convex, differentiable, and its gradient is
\begin{equation}
  \label{eq:grad-sq-distance}
  \nabla \phi (u) = u - \mathrm{P}_{C}(u),
\end{equation}
where~$\mathrm{P}_{C}(u)$ is the orthogonal projection of point~$u$
onto the set~$C$,
\begin{equation*}
  \mathrm{P}_{C}(u) = \argmin_{y \in C} \|u-y\|.
\end{equation*}
Knowing this, we can compute the gradient of~\eqref{eq:huber-norm} as
\begin{equation*}
  \begin{aligned}
    \nabla \psi_{R}(u) &=  2 \mathrm{P}_{R}(u)
  \end{aligned}
\end{equation*}
and the gradient of~\eqref{eq:matrix-cost} as
\begin{IEEEeqnarray}{rCl}\IEEEnonumber
  \nabla F (z) &=& \frac 12 B^{\top} \nabla  \psi_{\tilde{R}}(Bz) +
  \frac 12
  E^{\top} \nabla  \psi_{\tilde{R}a}(Ez-\alpha)\\\label{eq:gradient}
  &=& B^{\top} \mathrm{P}_{\tilde{R}}(Bz) + E^{\top}
  \mathrm{P}_{\tilde{R}a}(Ez-\alpha).
\end{IEEEeqnarray}

\subsection{Lipschitz constant}
\label{sec:lipschitz-constant}

It is widely known that projections onto convex sets shrink
distances~\cite{Phelps1957}, \ie,
\begin{equation*}
  \|\mathrm{P}_{C}(u) - \mathrm{P}_{C}(v)\| \leq \|u - v\|,
\end{equation*}
and this means that the gradient of~\eqref{eq:huber-norm} is Lipschitz
continuous with Lipschitz constant~$1$. Using this, we can
compute a Lipschitz constant for~\eqref{eq:gradient}. First, let us
focus on the inter-node term in~\eqref{eq:gradient}:
\begin{IEEEeqnarray*}{rCl}\IEEEnonumber
  \|B^{\top}\mathrm{P}_{\tilde{R}}(Bu) -
  B^{\top}\mathrm{P}_{\tilde{R}}(Bv)\| &\leq&
  \mtxnorm{B}\|\mathrm{P}_{\tilde{R}}(Bu) -
  \mathrm{P}_{\tilde{R}}(Bv)\|\\\IEEEnonumber
  &\leq&\mtxnorm{B}\|Bu-Bv\|\\\IEEEnonumber
  &\leq&\mtxnorm{B}^{2}\|u-v\|\\\IEEEnonumber
  &=& \lambda_{\mathrm{max}}(BB^{\top})\|u-v\|.
\end{IEEEeqnarray*}
The maximum eigenvalue of~$BB^{\top}$ can be bounded by
\begin{IEEEeqnarray}{rCl}\IEEEnonumber
  \lambda_{\mathrm{max}}(BB^{\top}) &=& \lambda_{\mathrm{max}}\left(
  \begin{bmatrix}
    A& -I& 0
  \end{bmatrix}
  \begin{bmatrix}
    A^{\top}\\ -I \\ 0
  \end{bmatrix}
\right)\\\IEEEnonumber
&=&\lambda_{\mathrm{max}}(AA^{\top} +I)\\\IEEEnonumber
&=& (1+ \lambda_{\mathrm{max}}(AA^{\top}))\\\IEEEnonumber
&=& (1+ \lambda_{\mathrm{max}}(L))\\ \label{eq:inter-node-lips}
&\leq& (1+ 2\delta_{\mathrm{max}}),
\end{IEEEeqnarray}
where~$L$ is the graph Laplacian, and~$\delta_{\max}$ is the maximum
node degree of the network. A proof of the
inequality~$\lambda_{\mathrm{max}}\leq 2\delta_{\mathrm{max}}$ is in
Bapat~\cite{Bapat2010}. In the same way, for the node-anchor
terms, we have
\begin{IEEEeqnarray}{rCl}\IEEEnonumber
  \|E^{\top}\mathrm{P}_{\tilde{R}a}(Eu) -
  E^{\top}\mathrm{P}_{\tilde{R}a}(Ev)\| &\leq&
  \mtxnorm{E}^{2}\|u-v\|\\\IEEEnonumber
  &=& \lambda_{\mathrm{max}}(EE^{\top})\|u-v\|\\\IEEEnonumber
\end{IEEEeqnarray}
and this constant can be upper-bounded by
\begin{IEEEeqnarray}{rCl}\IEEEnonumber  
\lambda_{\mathrm{max}}(EE^{\top})&=&\lambda_{\mathrm{max}}\left(
  \begin{bmatrix}
    E_{x} & 0& -I
  \end{bmatrix}
  \begin{bmatrix}
    E_{x}^{\top}\\ 0\\-I
  \end{bmatrix}
\right)\\\IEEEnonumber 
&=&
\lambda_{\mathrm{max}}\left(E_{x}E^{\top}_{x}+I\right)\\\IEEEnonumber
&\leq&
\left(1+\lambda_{\mathrm{max}}(E_{x}E^{\top}_{x})\right)\\\label{eq:node-anchor-lips}
&\leq& \left(1 + \max_{i \in \mathcal{V}}
  |\mathcal{A}_{i}|\right).
\end{IEEEeqnarray}
From~\eqref{eq:inter-node-lips} and~\eqref{eq:node-anchor-lips} we can
see that a Lipschitz constant for~\eqref{eq:matrix-cost} is
\begin{equation}
  \label{eq:lipschitz}
  L_{F} = 2 + 2\delta_{\mathrm{max}} + \max_{i \in \mathcal{V}} |\mathcal{A}_{i}|.
\end{equation}
We stress that this constant is small, does not depend on the size of
the network, and can be computed in a distributed
way~\cite{chung1997spectral}.

\subsection{Synchronous algorithm}
\label{sec:synchr-algor}

The gradient in~\eqref{eq:gradient} and its Lipschitz continuity, with
the constant in~\eqref{eq:lipschitz}, equip us to use the optimal
gradient optimization method due to Nesterov~(\cite{Nesterov1983}\cite{Nesterov2004}), further developed by Beck and Teboulle \cite{BeckTeboulle2009}.
Firstly, we must write the problem as an unconstrained minimization
using an indicator function~$I_{\mathcal{Z}}(u)=
\begin{cases}
  0 \qquad \text{if } u \in \mathcal{Z}\\ +\infty \quad \text{otherwise }
\end{cases}
$, and incorporate the constraints in the problem formulation. Then we
perform the proximal minimization of the unconstrained problem. The
result for our reformulation is shown in Algorithm~\ref{alg:synchronous}.
\begin{algorithm}[tb]
  \caption{Synchronous method: syncHuber}
  \label{alg:synchronous}
  \begin{algorithmic}[1] 
    \REQUIRE
    $L_{F}; \{d_{ij} : i \sim j \in \mathcal{E}\}; \{r_{ik} : i \in
    \mathcal{V}, k \in \mathcal{A}\};$
    \ENSURE $\hat x$ \STATE each node~$i$ chooses
    arbitrary~$x_{i}^{0} = x_{i}^{-1}$; \STATE set
    $y_{ij}^{0} = \mathrm{P}_{\mathcal{Y}_{ij}}\left(x^{0}_{i}
      -x^{0}_{j}\right)$,
    $ \mathcal{Y}_{ij} = \{y \in \reals^{p}: \|y\| \leq d_{ij}\} $;
    and
    $w_{ik}^{0} = \mathrm{P}_{\mathcal{W}_{ik}}\left(x^{0}_{i}
      -a_{k}\right)$,
    $ \mathcal{W}_{ik} = \{w \in \reals^{p}: \|w\| \leq r_{ik}\} $;
    \STATE $t = 0$; 
    \WHILE{some stopping criterion is not met, each
      node~$i$} 
    \STATE $t = t+1$; 
    \STATE
    $\xi_{i} = x_{i}^{t-1} + \frac{t-2}{t+1} \left(x_{i}^{t-1} -
      x_{i}^{t-2} \right)$;\label{alg:xi}
    \STATE broadcast~$\xi_{i}$ to all neighbors; 
    \FOR{all~$j$ in the neighbor set~$N_{i}$} 
    \STATE
    $\upsilon_{ij} = y_{ij}^{t-1} + \frac{t-2}{t+1} \left(y_{ij}^{t-1}
      - y_{ij}^{t-2} \right)$;\label{alg:upsilon}
    \STATE
    $ y_{ij}^{t} = \mathrm{P}_{\mathcal{Y}_{ij}} \left( \upsilon_{ij}
      + \frac{1}{L_{F}} \mathrm{P}_{R_{ij}}\left(\xi_{i}-\xi_{j} -
        \upsilon_{ij}\right)\right) $;
    \ENDFOR
    \FOR{all~$k$ in the anchor set~$\mathcal{A}_{i}$}
    \STATE
    $
    \omega_{ik} = w_{ik}^{t-1} + \frac{t-2}{t+1} \left(w_{ik}^{t-1} -
      w_{ik}^{t-2} \right)$;\label{alg:omega}
    \STATE  $w_{ik}^{t} = \mathrm{P}_{\mathcal{W}_{ik}}\left(
        \omega_{ik}+\frac{1}{L_{F}}\mathrm{P}_{Ra_{ik}}(\xi_{i}-a_{ik}-\omega_{ik})\right)
    $;
    \ENDFOR
    \STATE  
    \begin{align*}
        G = &\sum_{j \in N_{i}}
        \mathrm{P}_{R_{ij}}(\xi_{i}-\xi_{j}-\upsilon_{ij}) + \\
        & \sum_{k \in \mathcal{A}_{i}} \mathrm{P}_{Ra_{ik}}\left(\xi_{i}-a_{k}-\omega_{ik}\right);
      \end{align*}
    \STATE $
          x_{i}^{t} =\xi_{i} - \frac{1}{L_{F}}G
     $;\label{alg:x}
    \ENDWHILE
    \RETURN $\hat x_{i} = x_{i}^{t}$
  \end{algorithmic}
\end{algorithm}
Here,~$\mathcal{Y}_{ij}$ and~$\mathcal{W}_{ik}$ are the
sets~$\{x: \|x\| \leq d_{ij}\}$, and~$\{x: \|x\| \leq r_{ik}\}$,
respectively. Also,~$\mathcal{Y}$ and~$\mathcal{W}$ are the constraint
sets associated with the acquired measurements between sensors, and
between anchors and sensors, respectively, and~$N_{i}$ is the set of
 neighbor nodes of node~$i$. We denote the entries of~$\nabla F$
regarding variable~$x_{i}$ as~$G$. We observe that each block
of~$z = (x, y, w)$ at iteration~$t$ will only need local neighborhood
information, as shown in Algorithm~\ref{alg:synchronous}. To
demonstrate the natural distribution of the method we go back
to~\eqref{eq:gradient}. Here, the
term~$E^{\top} \mathrm{P}_{\tilde{R}a}(Ez-a)$ only involves anchor
measurements relative to each node, and so it is distributed. The
term~$B^{\top}\mathrm{P}_{\tilde{R}}(Bz)$ is less clear. The
vector~$Bz$ collects~$x_{i}-x_{j}-y_{ij}$ for all edges~$i \sim j$ and
to it we apply the projection operator onto the Cartesian product of
balls. This is the same as applying a projection of each edge onto
each ball. When left multiplying with~$B^{\top}$ we
get~$B^{\top}\mathrm{P}_{\tilde{R}}(Bz)$. The left multiplication
by~$B^{\top}$ will group at the position of each node variable~$x_{i}$
the contributions of all incident edges to node~$i$. To update
the~$y_{ij}$ variables we could designate one of the incident nodes as
responsible for the update and then communicate the result to the
non-computing neighbor. But, to avoid this expensive extra
communication, we decide that each node~$i$ should compute its
own~$y_{ij}$, where~$y_{ij} = -y_{ji}$ for all edges. Also, with this
device, the gradient entry regarding variable~$x_{i}$ would
be~$\sum_{j \in N_{i}} C_{(i\sim j,i)} \mathrm{P}_{R_{ij}}\left(
  C_{(i\sim j,i)}(x_{i}-x_{j} - y_{ij})\right)$.
The symbol $C_{(i \sim j,i)}$ denotes the arc-node incidence matrix
entry relative to edge~$i \sim j$ (row index) and node~$i$ (column
index).  As the projection onto a ball of radius~$R$ centered at the
origin can be written as~$\mathrm{P}_{R}(u) =
\begin{cases}
  \frac {u}{\|u\|}R & \qquad \text{if } \|u\| > R\\
  u & \qquad \text{if } \|u\| \leq R
\end{cases}, $
then~$\mathrm{P}_{R}(-u) = - \mathrm{P}_{R}(u)$, and, thus, the
gradient entry regarding variable~$x_{i}$
becomes~$\sum_{j \in N_{i}} \mathrm{P}_{R_{ij}}\left( x_{i}-x_{j} -
  y_{ij} \right)$, as stated in Algorithm~\ref{alg:synchronous}.
Each node~$i$ will update the current estimate of its own
position, each one of the~$y_{ij}$ for all the incident
edges and the anchor terms~$w_{ik}$, if any. 
In step~\ref{alg:xi} we have the extrapolation step for each~$x_{i}$,
whereas in steps~\ref{alg:upsilon} and~\ref{alg:omega} we can see the
update of the extrapolation steps for each one of the edge
variables~$y_{ij}$, and~$w_{ik}$, respectively.

\subsection{Asynchronous algorithm}
\label{sec:asynchr-algor}
In Section~\ref{sec:synchr-algor} we presented a distributed method
addressing the robust network localization problem in a scalable
manner, where each node uses information from its neighborhood and
performs a set of simple arithmetic computations. But the results
still depend critically on synchronous computation, where nodes
progress in lockstep through iterations. As the number of processing
nodes becomes very large, this synchronization can become seriously
difficult --- and unproductive. An asynchronous approach is called for
in such very large-scale and faulty settings. In an asynchronous time
model, the nodes move forward independently and algorithms withstand
certain types of faults, like temporary unavailability of a node.  To
address this issue, we present a fully asynchronous method, based on a
broadcast gossip scheme (c.f. Shah~\cite{shah2009} for an extended
survey of gossip algorithms).

Nodes are equipped with independent clocks ticking at random times
(say, as Poisson point processes).
When node~$i$'s clock ticks, it performs the update of its
variables and broadcasts the update to its neighbors. Let the
order of node activation be collected
in~$\{\chi_{t}\}_{t \in \naturals}$, a sequence of independent random
variables taking values on the set~$\mathcal{V}$, such that
\begin{equation}
  \label{eq:rv}
  \prob(\chi_{t} = i) = P_{i} > 0.
\end{equation}
Then, the asynchronous update of variables on node~$i$ can be
described as in Algorithm~\ref{alg:asyncronous}.
\begin{algorithm}[tb]
  \caption{Asynchronous method: asyncHuber}
  \label{alg:asyncronous}
  \begin{algorithmic}[1]
    \REQUIRE $L_{F}; \{d_{ij} : i \sim j \in \mathcal{E}\};
    \{r_{ik} : i \in \mathcal{V}, k \in \mathcal{A}\};$
    \ENSURE $\hat x$
    \STATE each node $i$ chooses random $x_{i}(0)$;
    \STATE set $y_{ij}^{0} =
    \mathrm{P}_{\mathcal{Y}_{ij}}\left(x^{0}_{i} -x^{0}_{j}\right)$,
    $
    \mathcal{Y}_{ij} = \{y \in \reals^{p}: \|y\| \leq d_{ij}\}
    $
    and $w_{ik}^{0} =
    \mathrm{P}_{\mathcal{W}_{ik}}\left(x^{0}_{i} -a_{k}\right)$,
    $
    \mathcal{W}_{ik} = \{w \in \reals^{p}: \|w\| \leq r_{ik}\}
    $
    \STATE $t = 0$;
    \WHILE{some stopping criterion is not met, each node $i$}
    \STATE $t = t + 1;$
    \STATE $
    \begin{aligned}[t]
      x_{i}^{t} =
      \begin{dcases}
        \argmin_{\substack{\xi_{i}, \{\upsilon_{ij} \in
            \mathcal{Y}_{ij}, i \sim j\}, \\ \{\omega_{ik} \in \mathcal{W}_{ik}, k
          \in \mathcal{A}_{k}\}}} F_{i}(\xi_{i},\{\upsilon_{ij}\},\{\omega_{ik}\}) \label{alg:asynminx}
        &\mbox{if } \chi_{t} = i\\
        x_{i}^{t-1} &\mbox{otherwise;}
      \end{dcases}
    \end{aligned}
    $
    \STATE if~$\xi_{t}= i$, broadcast~$x_{i}^{t}$ to neighbors
    \ENDWHILE
    \RETURN $\hat x = x^{t}$
  \end{algorithmic}
\end{algorithm}
To compute the minimizer in step~\ref{alg:asynminx} of
Algorithm~\ref{alg:asyncronous} it is useful to recast
Problem~\eqref {eq:cvx-matrix-z} as
\begin{equation}
  \label{eq:per-node}
  \begin{aligned}
    \minimize_{x,y,w} &\sum_{i} \left(\sum_{j \in N_{i}} \frac 14
      \|x_{i}-x_{j}-y_{ij}\|^{2} - \frac 14
      \mathrm{d}^{2}_{R_{ij}}(x_{i}-x_{j}-y_{ij}) + \right.\\ 
      & \left.\sum_{k \in \mathcal{A}_{i}}
      \frac 12 \|x_{i}-a_{k}-w_{ik}\|^{2} - \frac 12
      \mathrm{d}^{2}_{R_{ik}}(x_{i}-a_{k}-w_{ij}) \right)\\ 
      \st & \qquad y \in \mathcal{Y}\\
      &\qquad w \in \mathcal{W},
  \end{aligned}  
\end{equation}
where the factor~$\frac{1}{4}$ accounts for the duplicate terms when
considering summations over nodes instead of over edges. By fixing the
neighbor positions, each node solves a single source localization problem;
this setup leads to
\begin{equation}
  \label{eq:sl-problem}
  \begin{aligned}
    \minimize_{x_{i}, y_{ij},w_{ik}} & \; F_{i}(x_{i},\{y_{ij}, i \sim
    j\},\{w_{ik}, k \in \mathcal{A}_{i}\})\\
    \st & \qquad \|y_{ij}\| \leq d_{ij}\\
    & \qquad \|w_{ik}\| \leq r_{ik},
  \end{aligned}
\end{equation}
where
\begin{equation}
  \label{eq:sl-cost}
  \begin{aligned}
    F_{i}(x_{i},\{y_{ij}, i \sim j\},\{w_{ik}, k \in
    \mathcal{A}_{i}\}) = \\ \sum_{j \in N_{i}} \frac 14
    \|x_{i}-x_{j}-y_{ij}\|^{2} - \frac 14
    \mathrm{d}^{2}_{R_{ij}}(x_{i}-x_{j}-y_{ij}) + \\
     \sum_{k \in \mathcal{A}_{i}} \frac 12 \|x_{i}-a_{k}-w_{ik}\|^{2}
    - \frac 12 \mathrm{d}^{2}_{R_{ik}}(x_{i}-a_{k}-w_{ij}).
  \end{aligned}
\end{equation}
Problem~\eqref{eq:sl-problem} is convex, solvable at each
node by a general purpose solver. Nevertheless, that approach would not
take advantage of the specific problem structure, thus depriving the
solution of an efficient and simpler computational procedure. Again,
Problem~\eqref{eq:sl-problem} can be solved by the Nesterov optimal
first order method, because the gradient of~$F_{i}$ is Lipschitz
continuous in~$x_{i}, \{y_{ij}\}$, and ~$\{w_{ik}\}$, accepting the
same Lipschitz constant as~$F$, in~\eqref{eq:lipschitz}.

\section{Convergence analysis}
\label{sec:analysis}
In this section we address the convergence of 
Algorithms~\ref{alg:synchronous} and~\ref{alg:asyncronous}. We
provide convergence guarantees and rate of convergence for the
synchronous version, and we also prove convergence for the
asynchronous method.
\subsection{Synchronous algorithm}
\label{sec:analysis-synchr-algor}
As shown in Section~\ref{sec:distr-robust-sens}, Problem~\eqref
{eq:cvx-matrix-z} is convex and the cost function has a Lipschitz
continuous gradient. As proven by
Nesterov~(\hspace{1sp}\cite{Nesterov1983,Nesterov2004}), and further
developed by Beck and Teboulle~\cite{BeckTeboulle2009},
Algorithm~\ref{alg:synchronous} converges at the optimal rate
$O\left( t^{-2} \right)$; specifically,
$F( z^{t} ) - F^\star \leq \frac{2 L_{F}}{(t+1)^2} \left\| z^{0} -
  z^\star \right\|^2$,
where~$F^{\star}$ is the optimal value and~$z^{\star}$ is a minimizer
of Problem~\eqref {eq:cvx-matrix-z}.

\subsection{Asynchronous algorithm}
\label{sec:analysis-asynchr-algor}

To investigate the convergence of Algorithm~\ref{alg:asyncronous}, we
need the following assumptions:
\begin{assumption}
  \label{th:connected-assumption}
The topology of the network conforms to:
  \begin{itemize}
  \item The graph~$\mathcal{G}$ is connected;
  \item There is at least one node in~$\mathcal{G}$ with an anchor
    measurement.
  \end{itemize}
\end{assumption}
These assumptions are naturally fulfilled in the network localization
problem: if the network is supporting several disconnected components,
then each can be treated as a different network, and, for
disambiguation, localization requires the availability of 3 anchors in
2D and 4 anchors in 3D.

The convergence of Algorithm~\ref{alg:asyncronous} is stated next.
\begin{theorem}[Almost sure convergence]
  \label{th:asyn-convergence}
  Let Assumption~\ref{th:connected-assumption} hold. Consider
  Problem~\eqref{eq:cvx-matrix-z}, and the
  sequence~$\{z^{t}\}_{t \in \naturals}$ generated by
  Algorithm~\ref{alg:asyncronous}. Define the solution set as
  $\mathcal{Z}^{\star} = \{z \in \mathcal{Z}: F(z) =
  F^{\star}\}$. Then
  \begin{enumerate}
  \item $\mathrm{d}_{\mathcal{Z}^{\star}}(z^{t}) \to 0$, a.s.;
  \item $F(z^{t}) \to F^{\star}$, a.s.
  \end{enumerate}
\end{theorem}
The reader can find the proof in Appendix~\ref{sec:proof-theor-refth}.
It is also possible to state that, with probability one,
Algorithm~\ref{alg:asyncronous} converges in a \emph{finite} number of
iterations. The result is stated in the following theorem, also proven
in Appendix~\ref{sec:proof-theor-refth}.
\begin{theorem}
\label{th:nr-iter-convergence}
For a prescribed precision~$\epsilon$, the sequence of
iterates~$\{z^{t}\}_{t \in \naturals}$ converges in~$K_{\epsilon}$
iterations. The expected value of this number of iterations is
  \begin{equation}
    \label{eq:nr-iter-convergence}
    \expect \left[ K_{\epsilon} \right] \leq \frac{F\left(z^{0}\right)
    - F^{\star}}{b_{\epsilon}},
  \end{equation}
where~$b_{\epsilon}$ is a constant that depends on the
specified~$\epsilon$.
\end{theorem}

\section{Numerical experiments}
\label{sec:numer-exper}

\subsection{Underestimator performance}
\label{sec:under-perf}

We assess the performance of the three considered loss functions
through simulation. The experimental setup consists in a uniquely
localizable geometric network deployed in a square area with side
of~$1$Km, with four anchors (blue squares in
Figure~\ref{fig:estimates}) located at the corners, and ten sensors,
(red stars). Measurements are also visible as dotted green lines. The
average node degree\footnote{To characterize the network we use the
  concepts of \emph{node degree}~$k_{i}$, which is the number of edges
  connected to node~$i$, and \emph{average node degree}~$\langle k \rangle = 1/n
  \sum_{i=1}^{n}k_{i}$.} of the network is~$4.3$.  The regular noisy
range measurements are generated according to
\begin{align}
  \nonumber
  d_{ij} = | \|x_{i}^{\star} - x_{j}^{\star}\| + \nu_{ij} |,
  \\ \label{eq:noise}
  r_{ik} = | \|x_{i}^{\star} - a_{k}\| + \nu_{ik} |,
\end{align}
where~$x_{i}^{\star}$ is the true position of node~$i$,
and~$\{\nu_{ij} : i \sim j \in \mathcal{E}\} \cup \{\nu_{ik} : i \in
\mathcal{V}, k \in \mathcal{A}_{i}\}$
are independent Gaussian random variables with zero mean and standard
deviation~$0.04$, corresponding to an uncertainty of about~$40$m.
Node~$7$ is malfunctioning and all measurements related to it are
corrupted by Gaussian noise with standard deviation~$4$, corresponding
to an uncertainty of~$4$Km. The convex optimization problems were
solved with \texttt{cvx}~\cite{cvx}.  We ran~$100$ Monte Carlo trials,
sampling both regular and outlier noise.

The performance metric used to assess accuracy is the 
positioning error per sensor defined as
\begin{equation}
  \label{eq:error}
  \epsilon(m) =  \frac{\|\hat x(m) - x^{\star}\|}{|\mathcal{V}|},
\end{equation}
where~$\hat x(m)$ corresponds to the position estimates for all
sensors in Monte Carlo trial~$m$. The empirical mean of the
positioning error is defined as
\begin{equation}
  \label{eq:avgerror}
  \epsilon = \frac1{M} \sum_{m =1}^{M} \epsilon(m),
\end{equation}
where~$M$ is the number of Monte Carlo trials.
\begin{figure}[tb]
  \centering
  \includegraphics[width=\columnwidth]{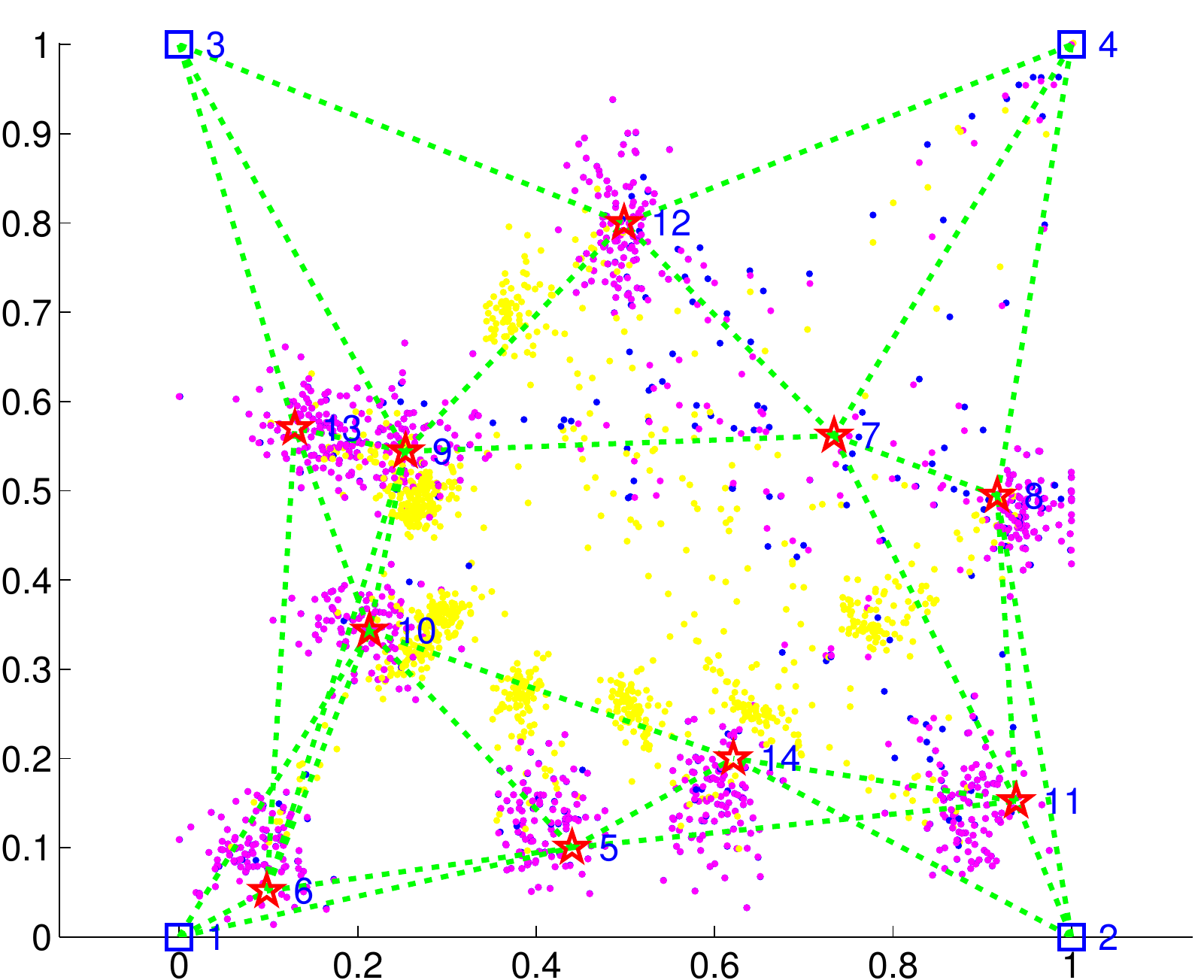}
  \caption{Underestimator performance: Estimates of sensor positions
    for the three loss functions; We plotted in yellow the monte carlo
    results of minimizing~$f_{|\cdot|}$, the $L_{1}$ loss; in blue we
    can see the estimates resulting from minimizing~$f_{Q}$, the
    quadratic loss; in the same way, magenta dots represent the output
    for function~$f_{R}$, the Huber loss. It is noticeable that
    the~$L_{1}$ loss is not able to correctly estimate positions whose
    measurements are corrupted with Gaussian noise. The perturbation
    in node~$7$ has more impact in the dispersion of blue
    dots than magenta dots around its neighbors.}
  \label{fig:estimates}
\end{figure}
In Figure~\ref{fig:estimates} we can observe that clouds of estimates
from~$g_{R}$ and~$g_{Q}$ gather around the true positions, except for
the malfunctioning node~$7$. Note the increased spread of blue dots
around nodes with edges connecting to node~$7$, indicating that~$g_{R}$
better preserves the nodes' ability to localize themselves, despite
their confusing neighbor, node~$7$.
\begin{figure}[tb]
  \centering
  \includegraphics[width=\columnwidth]{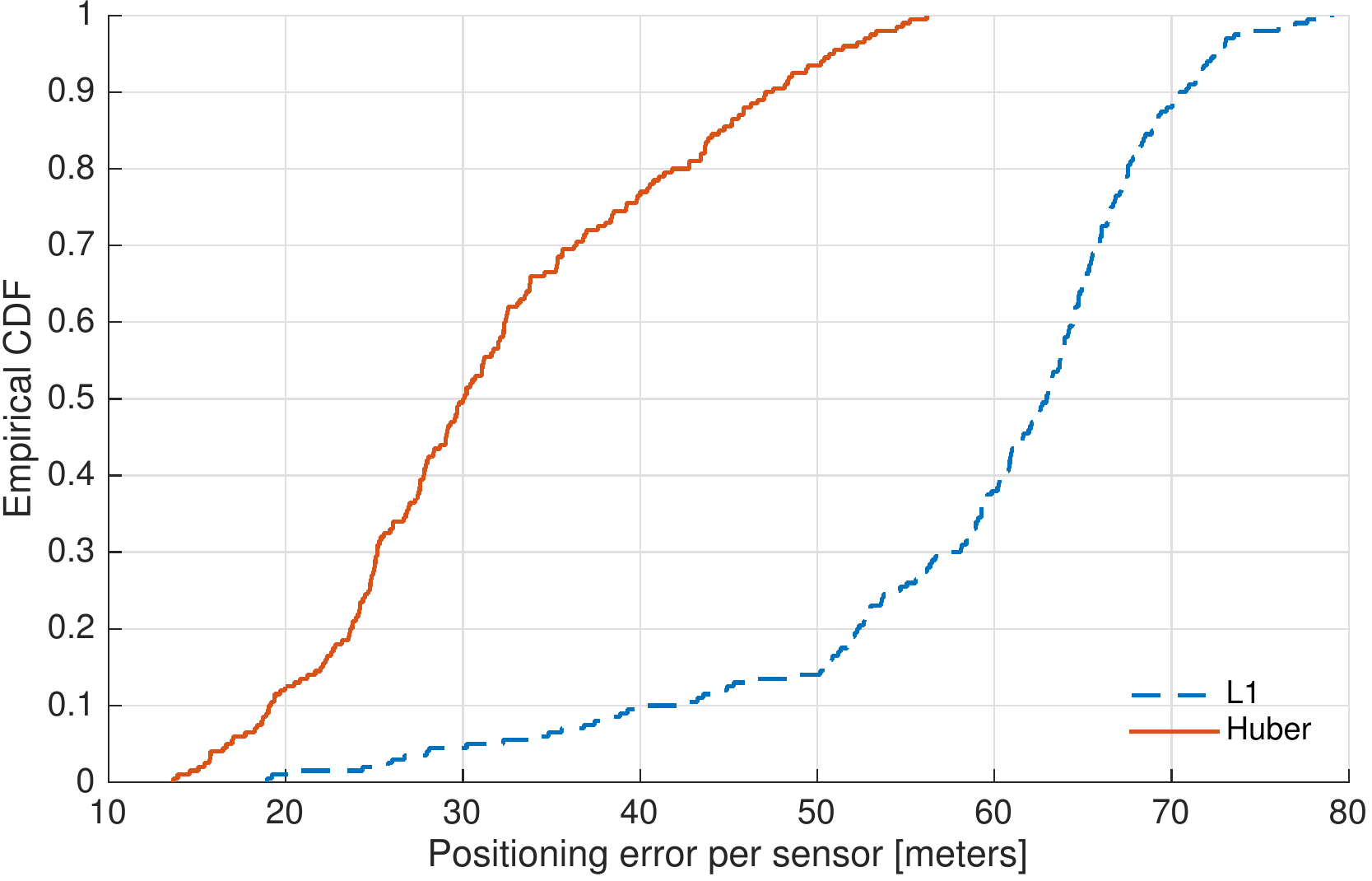}
  \caption{Underestimator performance: Empirical CDF for the
    positioning error per sensor, in meters, for the Gaussian outlier
    noise experiment.}
  \label{fig:error}
\end{figure}
This intuition is confirmed by the empirical CDFs of estimation errors
shown in Figure~\ref{fig:error}, which demonstrate that the Huber
robust cost can reduce the error per sensor by an average of~$28.5$
meters, when compared with the~$L_{1}$ discrepancy.  Also, as
expected, the malfunctioning node cannot be positioned by any of the
algorithms.
\begin{figure}[tb]
  \centering
  \includegraphics[width=\columnwidth]{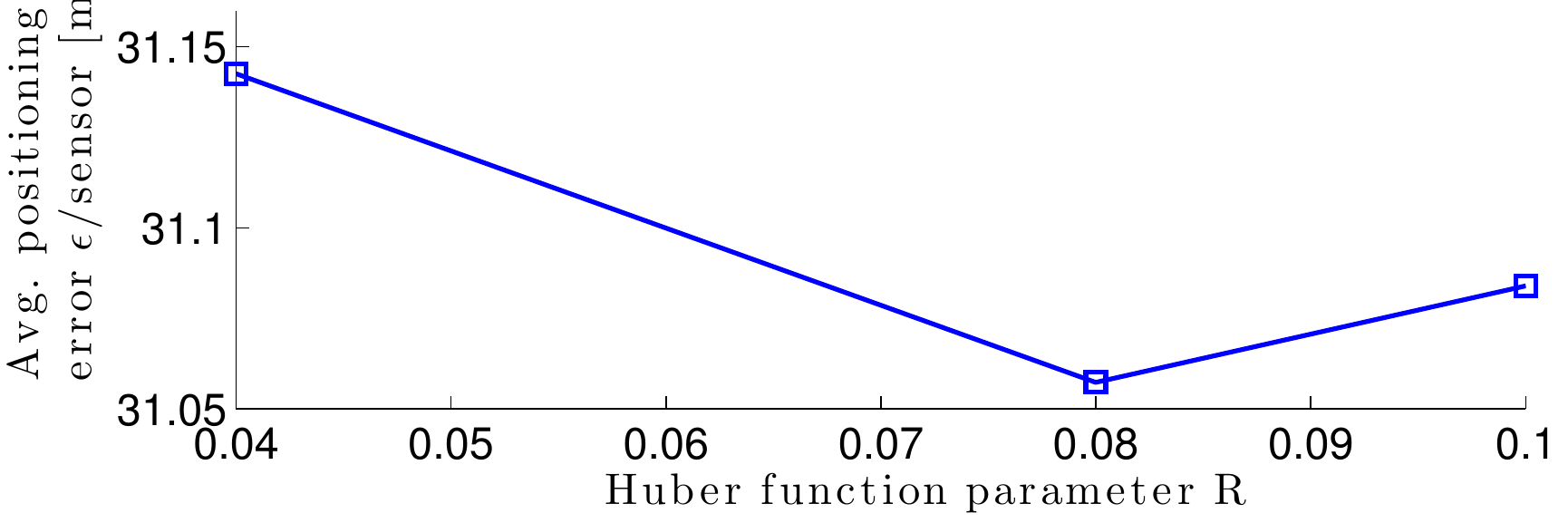}
  \caption{Underestimator performance: Average positioning error
    \textit{versus} the value of the Huber function parameter~$R$.
    Accuracy is maintained for a wide range of parameter values.}
  \label{fig:varR}
\end{figure}
The sensitivity to the value of the Huber parameter~$R$
in~\eqref{eq:huber-loss} is only moderate, as shown in
Figure~\ref{fig:varR}. In fact, the error per sensor of the proposed
estimator is always the smallest for all tested values of the
parameter. We observe that the error increases when~$R$ approaches the
standard deviation of the regular Gaussian noise, meaning that the
Huber loss gets closer to the~$L_{1}$ loss and, thus, is no longer
adapted to the regular noise ($R=0$ corresponds exactly to the~$L_{1}$
loss); in the same way, as~$R$ increases, so does the quadratic
section, and the estimator gets less robust to outliers, so, again,
the error increases.

Another interesting experiment is to see what happens when the faulty
sensor produces measurements with consistent errors or bias. We
ran~$100$ Monte Carlo trials in the same setting, but node~$7$
measurements are now consistently~$10\%$ of the real distance to each
neighbor.
\begin{figure}[tb]
  \centering
  \includegraphics[width=\columnwidth]{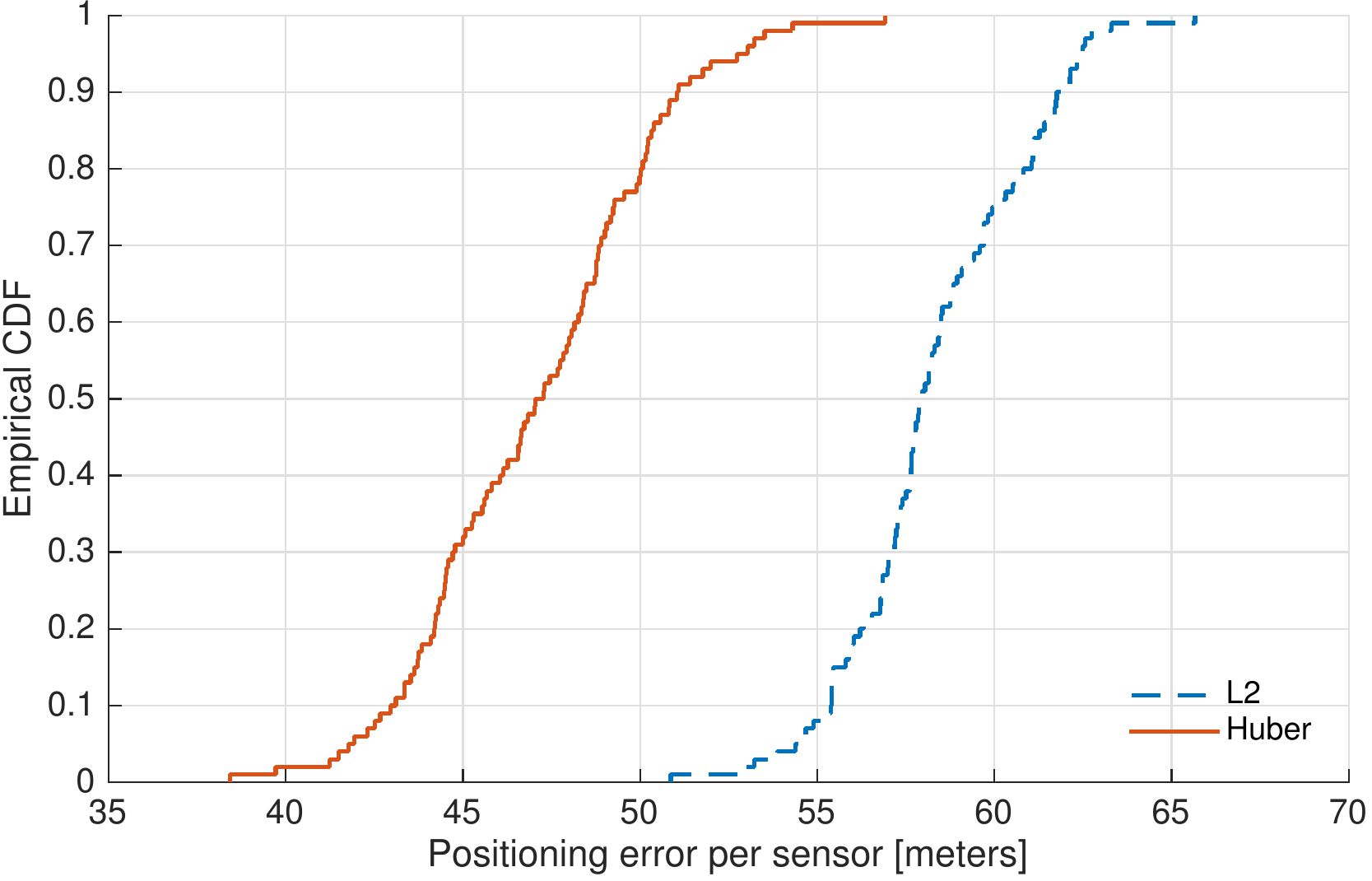}
  \caption{Underestimator performance: Empirical CDF for the
    positioning error per sensor, in meters, for the biased
    node experiment.}
  \label{fig:error-bias}
\end{figure}
The empirical CDF for the positioning error per sensor is shown in
Figure~\ref{fig:error-bias}. Here we observe a significant performance
gap between the alternative costs --- in average about~$11.2$ meters
--- so the Huber formulation proves to be superior even with biased
sensors.

\subsection{Performance of the distributed synchronous
  Algorithm~\ref{alg:synchronous}}
\label{sec:perf-distr-algor}

We tested Algorithm~\ref{alg:synchronous} using the same setup as in
the previous section, with node~$7$ contaminated with added Gaussian
noise with standard deviation of~$4$.
\begin{figure}[tb]
  \centering
  \includegraphics[width=\columnwidth]{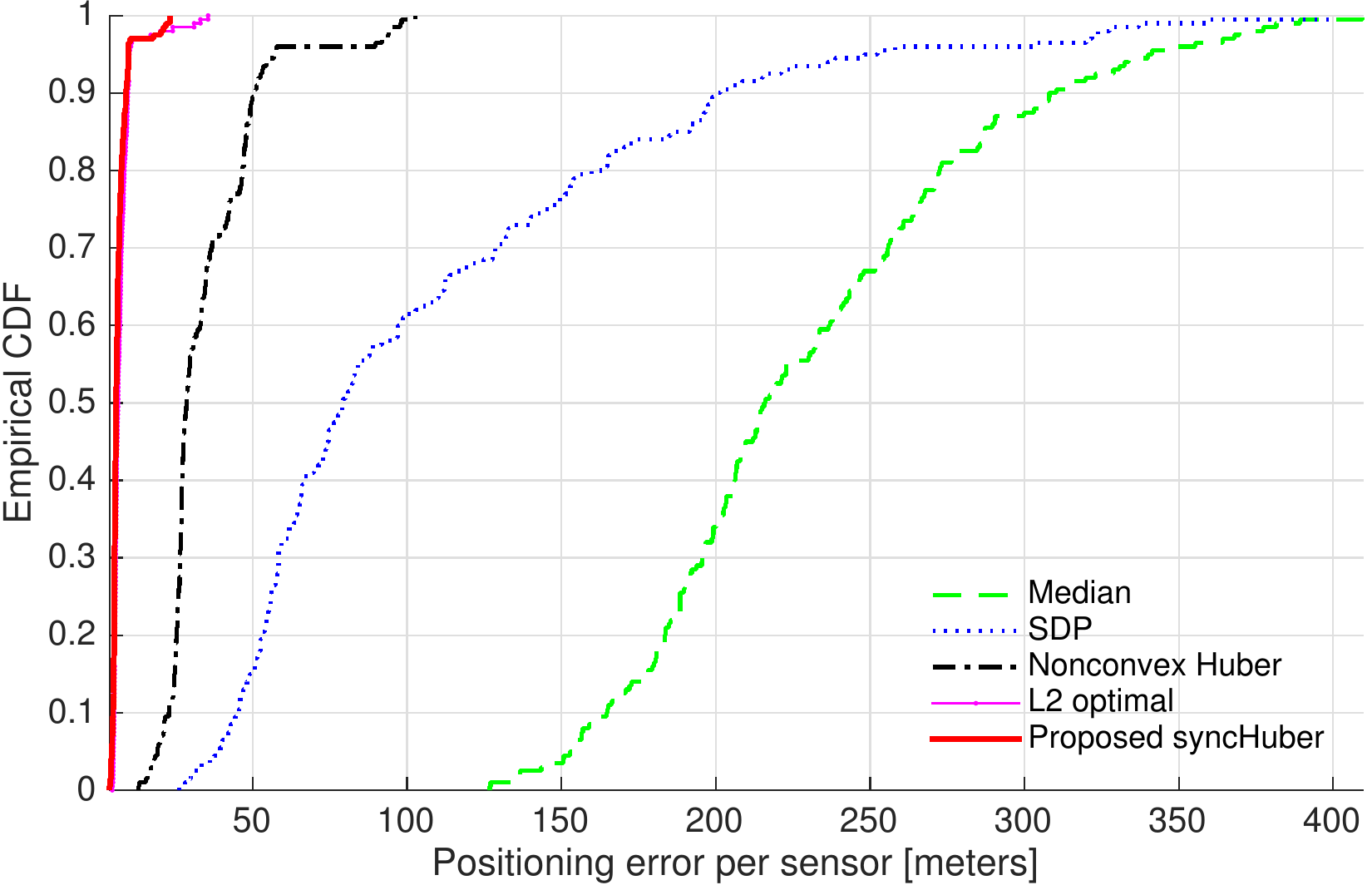}
  \caption{Accuracy of the distributed synchronous algorithm:
    Empirical CDF for the positioning error per sensor, in meters, for
    the Gaussian outlier noise experiment, discarding the positioning
    error of the malfunctioning node.}
  \label{fig:cdf-distr-median-SDP-L2-H-Gauss-noise}
\end{figure}
We benchmark our method comparing with the performance of the
centralized solutions in O\u{g}uz-Ekim
\ea,~\cite{OguzGomesXavierOliveira2011}, which we denote as ``Median''
below, the SDP presented by Simonetto and
Leus\footnote{In~\cite{SimonettoLeus2014}, the authors present a
  distributed ESDP algorithm which is a relaxation of the centralized
  SDP. As the simulation time for the distributed, edge-based
  algorithm is considerable we benchmarked against the tighter and
  more accurate centralized SDP solution.}~\cite{SimonettoLeus2014}, and
also the distributed locally convergent algorithm by Korkmaz and Van der
Veen\footnote{This distributed method attacks directly the nonconvex
  cost~\eqref{eq:huberDist}, thus delivering a local solution,
  that depends on the initialization point. The algorithm was initialized
  with Gaussian noise.}~\cite{KorkmazVeen2009}. The results are
summarized in Figure~\ref{fig:cdf-distr-median-SDP-L2-H-Gauss-noise}. Here
the empirical CDFs of the positioning error~\eqref{eq:error} show a
superior accuracy of our syncHuber algorithm.

When analyzing the results for the biased experiment as described in
the previous section, it is noticeable that the syncHuber algorithm
beats the
state-of-the-art~\cite{soares2014simple} for the quadratic
discrepancy by more than 5 meters per sensor in average positioning
error, as depicted in Figure~\ref{fig:cdf-distr-SDP-L2-H}.
\begin{figure}[tb]
  \centering
  \includegraphics[width=\columnwidth]{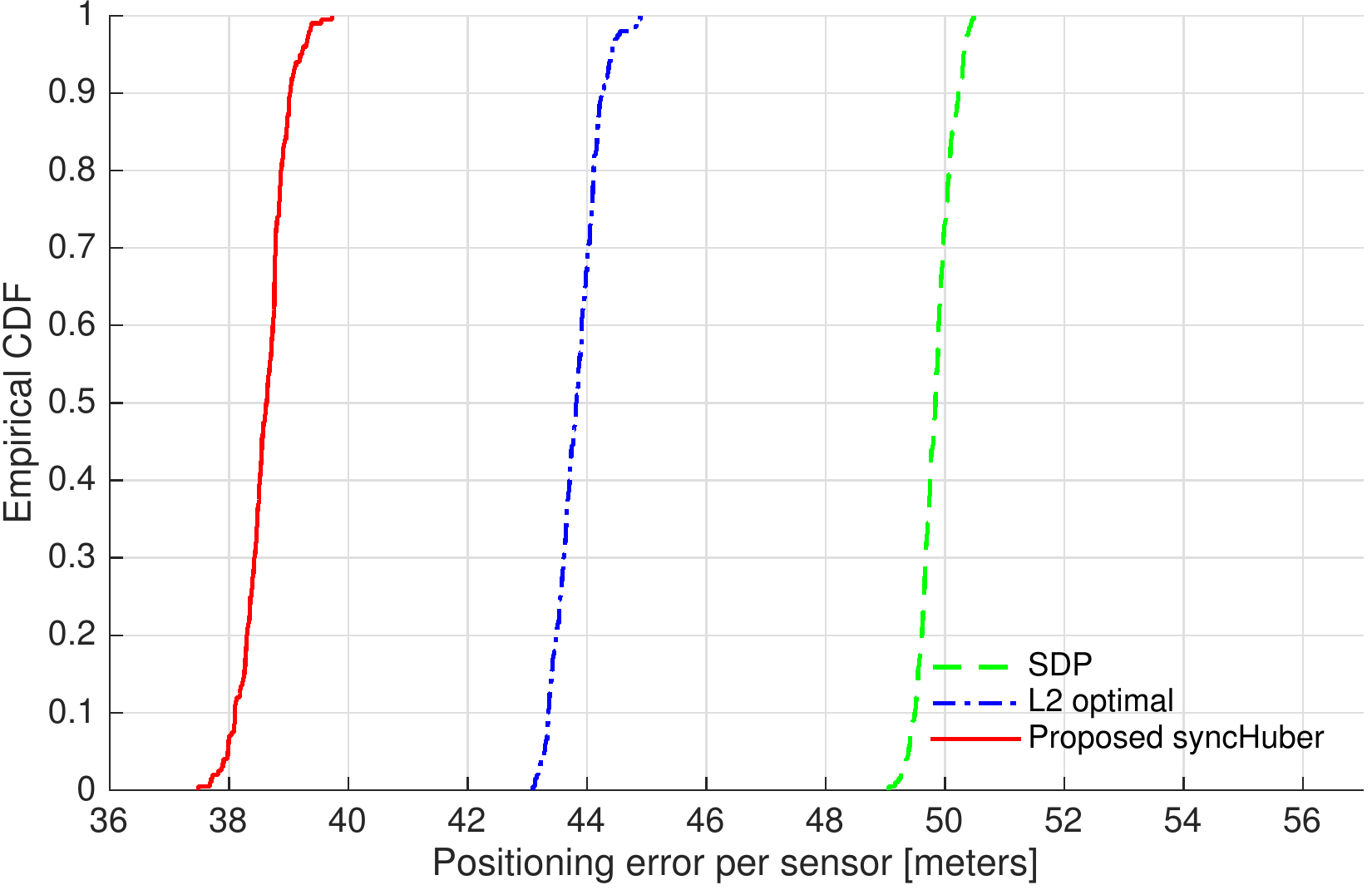}
  \caption{Accuracy of the distributed synchronous algorithm:
    Empirical CDF for the positioning error per sensor, in meters, for
    the biased node experiment discarding the positioning error of the
    malfunctioning node.}
  \label{fig:cdf-distr-SDP-L2-H}
\end{figure}
As expected from the results in the previous section, when we compare
to a~$L_{1}$-type algorithm --- in this case the ``Median'' from
O\u{g}uz-Ekim \ea~\cite{OguzGomesXavierOliveira2011} ---
\begin{figure}[tb]
  \centering
  \includegraphics[width=\columnwidth]{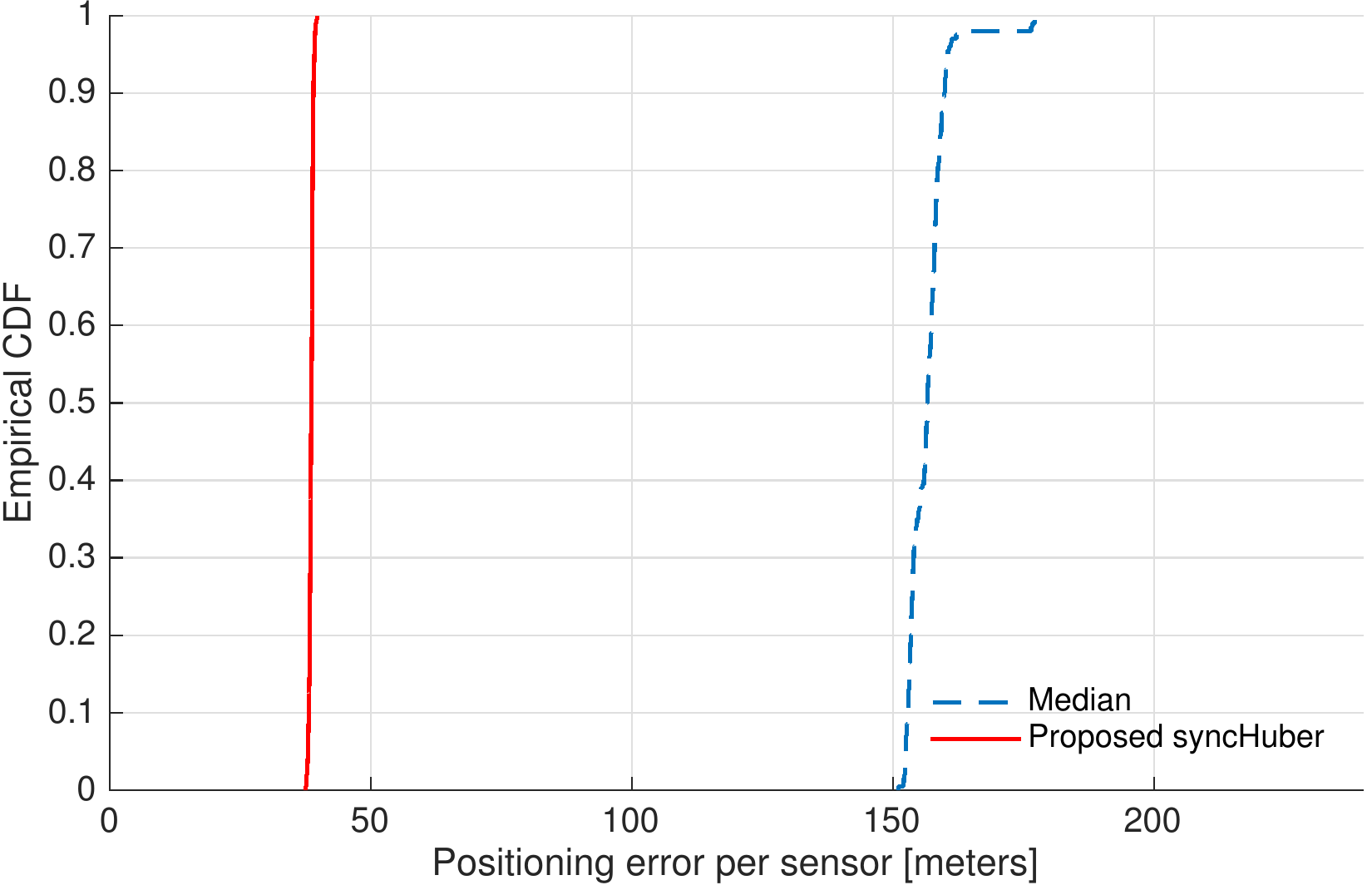}
  \caption{Accuracy of the distributed synchronous algorithm:
    Empirical CDF for the positioning error per sensor, in meters, for
    the biased experiment di
scarding the positioning error of the
    malfunctioning node.}
  \label{fig:cdf-distr-median-H}
\end{figure}
the improvement of performance of our solution is outstanding (on
average about 120 metrs per sensor), as depicted in
Figure~\ref{fig:cdf-distr-median-H}.

\subsection{Performance of the distributed asynchronous
  Algorithm~\ref{alg:asyncronous}}
\label{sec:perf-asynchr-algor}

Here, we tested Algorithm~\ref{alg:asyncronous}, asyncHuber, using the
same setup as in the previous sections, with node~$7$ contaminated
with added Gaussian noise with standard deviation of~$4$. We
benchmarked it against the synchronous
Algorithm~\ref{alg:synchronous}, syncHuber, since both minimize the
same cost function. The algorithms were allowed to run with the same
communication load.
\begin{figure}[tb] 
  \centering
  \includegraphics[width=\columnwidth]{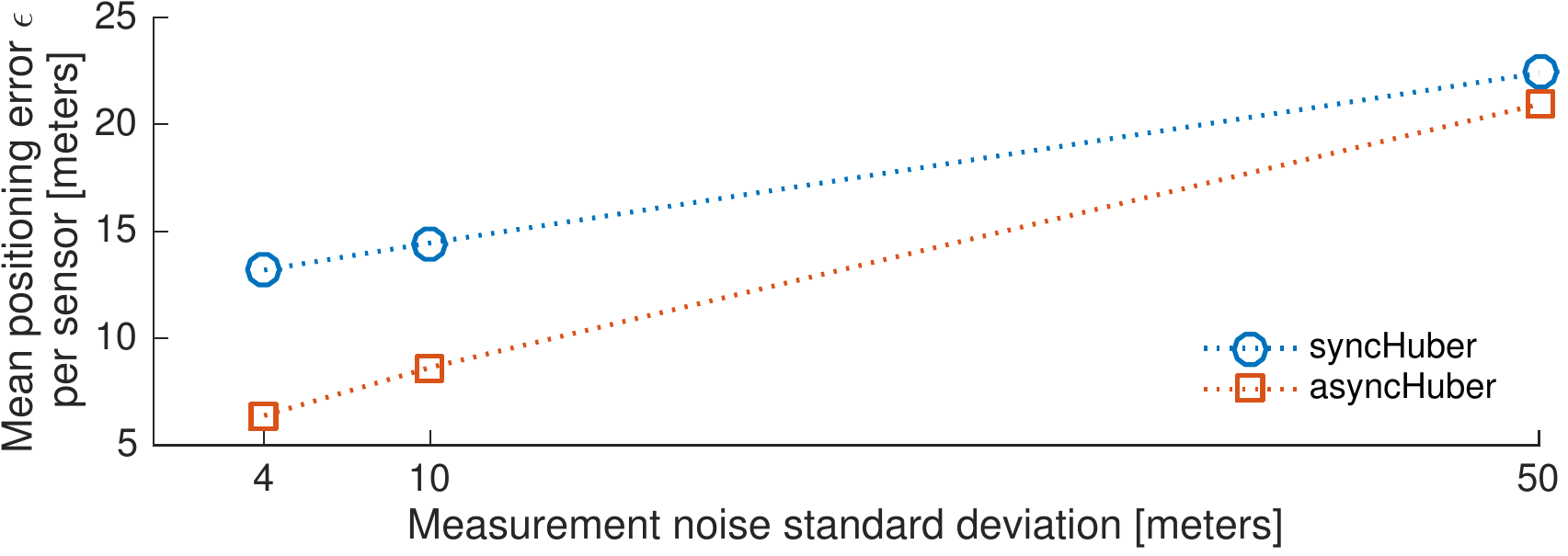}
  \caption{Accuracy of the distributed asynchronous algorithm: Mean
    positioning error of synchronous algorithm~\ref{alg:synchronous}
    versus asynchronous algorithm~\ref{alg:asyncronous}, discarding
    the positioning error of the malfunctioning node. Both algorithms
    were run with the same communication load.}
  \label{fig:syncVSasync}
\end{figure}
The mean positioning error for the considered noise levels is depicted
in Figure~\ref{fig:syncVSasync}. We can observe that the asyncHuber
algorithm fares better than syncHuber for the same communication
amount. This is an interesting phenomenon empirically observed in
different optimization algorithms when comparing deterministic and
randomized versions. In fact, Bertsekas and
Tsitsiklis~\cite[Section~6.3.5]{BertsekasTsitsiklis1989} provide a
proof of this behavior for a restricted class of algorithms.
\begin{figure}[tb] 
  \centering
  \includegraphics[width=\columnwidth]{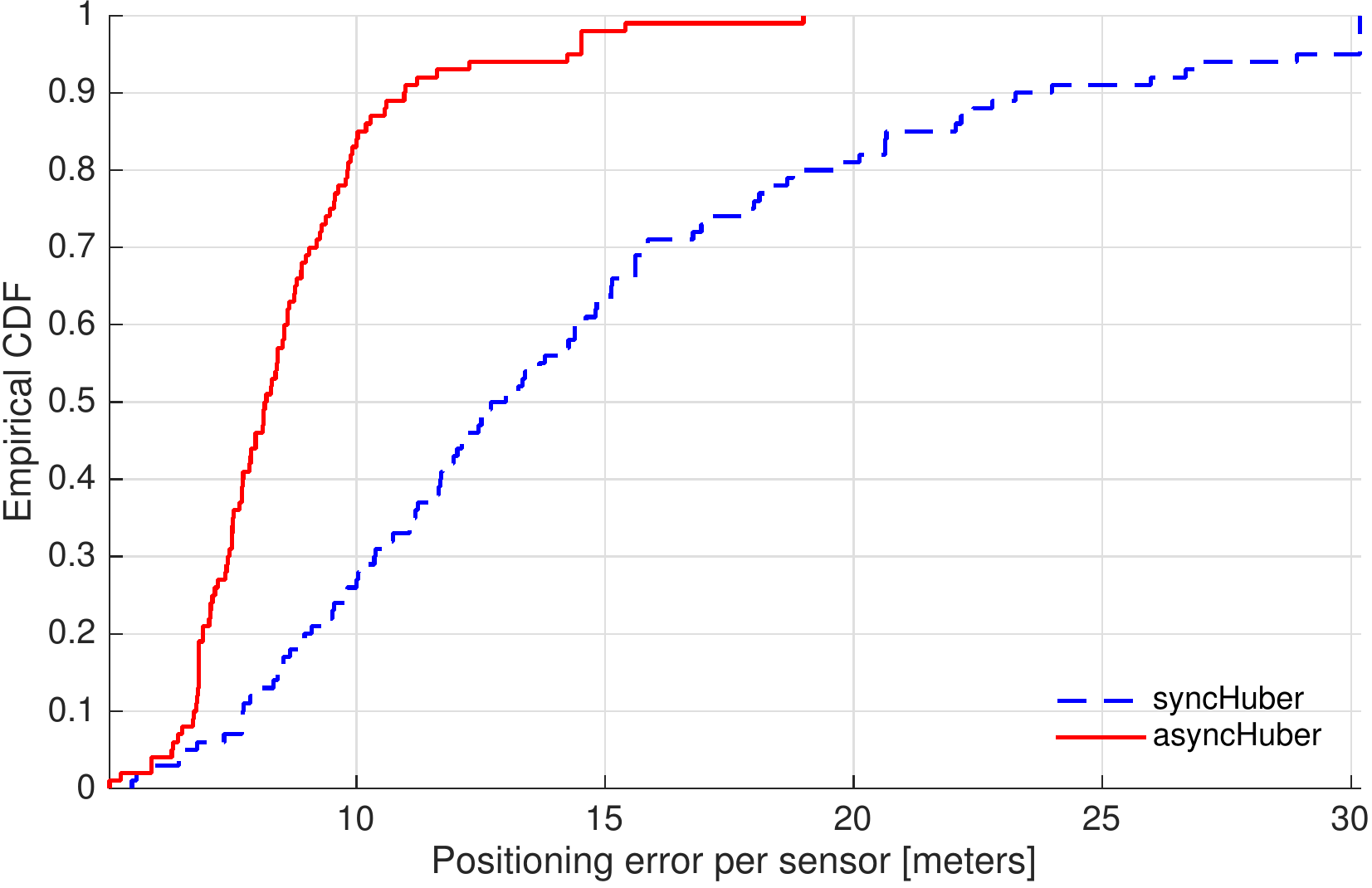}
  \caption{Accuracy of the distributed asynchronous algorithm: CDF of the
    positioning error of synchronous
    algorithm~\ref{alg:synchronous} versus asynchronous
    algorithm~\ref{alg:asyncronous}, discarding the positioning error
    of the malfunctioning node. Both algorithms were run with the same
    communication load. Experiment with measurements contaminated by
    medium power noise ($\sigma=0.01$), corresponding to 10 meters of
    standard deviation for a square with 1 Km sides.}
  \label{fig:cdf-sync-vs-async-sd2}
\end{figure}
Figure~\ref{fig:cdf-sync-vs-async-sd2} further explores the
experimental data, by examining the CDF of the positioning error for
the tested Monte Carlo trials. Here we see the superior accuracy of
the asynchronous Huber Algorithm~\ref{alg:asyncronous}, for the same
communications volume.  We must, nevertheless, emphasize that this
result does not correspond to a faster algorithm, in terms of running
time: syncHuber in one iteration updates all of the nodes positions in
parallel, and broadcasts the current estimates across neighbors,
whereas in asyncHuber only one node operates at a time. As the
wireless medium might be much more intensively used for synchronous
updates than for random gossip interactions, it seems entirely
possible that for the same operation time syncHuber will outperform
asyncHuber --- at the expense of greater overall power consumption.

\section{Discussion and conclusions}
\label{sec:concl-future-work}

We presented two distributed, fast, and robust localization algorithms
that take noisy ranges and a few anchor locations, and output accurate
estimates of the node positions. We approximated the difficult,
nonconvex problem based on the Huber discrepancy
in~\eqref{eq:snlOptProb} with a convex envelope of terms, robust to
outliers. How does the
Huber-based approximation in~\eqref{eq:huber-cvx} compares with
similar $L_{1}$ and $L_{2}$ underestimators, frequent in robust
estimation contexts?  A smaller optimality gap means a more robust
approximation~\cite{DestinoAbreu2011}: We designed a bound that
certifies the gap between the nonconvex and surrogate optimal values
for Huber, $L_{1}$ and $L_{2}$ and shows a tighter gap in the Huber
case. A numerical analysis of a star network in 1D unveiled that the
optimality gap for the Huber approximation was one order of magnitude
less than the quadratic or absolute value convexified problems, with
respect to their nonconvex counterparts. Numerical network
localization trials verify the surrogate robust behavior under
different types of outlier noise.
In order to develop a distributed method we needed to transform our
cost. So, we proposed a new representation of the Huber function
composed with a norm, and arrived at a novel distributed gradient
method, syncHuber, with optimal convergence rate. But our syncHuber
algorithm requires synchronization, which is a difficult demand for
many applications. Thus, we put forward a novel asynchronous method
for robust network localization, asyncHuber, converging with
probability one. Nevertheless, like any other relaxation method,
ours are prone to the anchor convex hull problem: preliminary results
show the positioning accuracy degrades --- albeit graciously --- when
nodes' positions depart from the anchor convex hull. Arguably, this is
not a big issue because engineers in general can control the choice or
placement of anchoring landmarks, and can delimit the area under
survey.

In sum, both our algorithms work with simple computations at each node
and minimal communication payloads, have provable convergence and show
a superior performance in our numerical experiments. Also, they do not
require knowledge of a Hamiltonian path in the network, which
simplifies real-world implementation, unlike the method presented by
Yousefi \ea~\cite{YousefiChangChampagne2014}. In average, the
positioning error of Algorithm~\ref{alg:synchronous} is less 120m for
a deployment in a square of 1Km sides than the state-of-the-art
$L_{1}$ centralized method for robust network localization of
O\u{g}uz-Ekim~\ea~\cite{OguzGomesXavierOliveira2011}.

\section*{Acknowledgment}
The authors would like to thank Pinar O\u{g}uz-Ekim and Andrea
Simonetto for providing the MATLAB implementations of their published
algorithms. Also, we thank Jo\~{a}o Xavier for the interesting
discussions during this research.

\appendices

\section{Proof of Proposition~\ref{prop:cvx-term-norm}}
\label{sec:proof-prop-cvx-term-norm}
We define a function~$\phi(u) = \max\{0,h_{R_{ij}}(u)\}$, and restate
a generic term of the first summation in~\eqref{eq:huber-cvx}
\begin{equation*}
  h_{R_{ij}}(s(\|x_{i}-x_{j}\| - d_{ij}))
\end{equation*}
as
\begin{equation*}
 \phi(s(\|x_{i}-x_{j}\| - d_{ij})),
\end{equation*}
which represents the same mathematical object,
because~$s(\|x_{i}-x_{j}\| - d_{ij})$ is always nonnegative.
Now, we prove the equivalence relation~\eqref{eq:cvx-term-norm},
beginning by 
\begin{equation}
  \label{eq:proof-cvx-term-inf}
  \phi(s(\|x_{i}-x_{j}\| - d_{ij})) \leq \inf_{\|y_{ij}\| \leq d_{ij}} \phi(\|x_{i}-x_{j}-y_{ij}\|).
\end{equation}
We choose any~$\bar{y}_{ij} : \|\bar{y}_{ij}\| \leq d_{ij}$, and note that
\begin{equation*}
  \begin{aligned}
    s(\|x_{i}-x_{j}\| - d_{ij}) &= \inf_{\|y_{ij}\| \leq d_{ij}}
    \|x_{i}-x_{j}-y_{ij} \|\\
&\leq \|x_{i}-x_{j}-\bar{y}_{ij} \|.
  \end{aligned}
\end{equation*}
As~$\phi$ is nondecreasing, 
\begin{equation*}
  \phi(s(\|x_{i}-x_{j}\| - d_{ij})) \leq \phi(\|x_{i}-x_{j}-\bar{y}_{ij} \|),
\end{equation*}
for all~$\bar{y}_{ij}$, and in particular,
 \begin{equation*}
  \phi(s(\|x_{i}-x_{j}\| - d_{ij})) \leq \inf_{\|y_{ij}\| \leq d_{ij}}\phi(\|x_{i}-x_{j}-{y}_{ij} \|),
\end{equation*}
which proves~\eqref{eq:proof-cvx-term-inf}.
We now establish 
\begin{equation}
  \label{eq:proof-cvx-term-inf-2}
  \phi(s(\|x_{i}-x_{j}\| - d_{ij})) \geq \inf_{\|y_{ij}\| \leq d_{ij}} \phi(\|x_{i}-x_{j}-y_{ij}\|).
\end{equation}
We choose~$y_{ij}^{\star}$, a minimizer of the optimization problem in
the RHS of~\eqref{eq:proof-cvx-term-inf-2}. We know that
\begin{equation*}
  s(\|x_{i}-x_{j}\| - d_{ij}) = \|x_{i}-x_{j}-y^{\star}_{ij}\|
\end{equation*}
and so
\begin{equation*}
  \phi(s(\|x_{i}-x_{j}\| - d_{ij})) = \phi(\|x_{i}-x_{j}-y^{\star}_{ij}\|),
\end{equation*}
and as~$\phi$ is monotonic,
\begin{equation*}
  \phi(\|x_{i}-x_{j}-y^{\star}_{ij}\|) \leq \phi(\|x_{i}-x_{j}-y_{ij}\|),
\end{equation*}
for all~$y_{ij} : \| y_{ij}\| \leq d_{ij}$. In particular,
\begin{equation*}
   \phi(\|x_{i}-x_{j}-y^{\star}_{ij}\|) \leq \inf_{\|y_{ij}\| \leq d_{ij}} \phi(\|x_{i}-x_{j}-y_{ij}\|),
\end{equation*}
which proves~\eqref{eq:proof-cvx-term-inf-2}, and concludes the proof
of Proposition~\ref{prop:cvx-term-norm}.

\section{Proofs of Theorems~\ref{th:asyn-convergence} and~\ref{th:nr-iter-convergence}}
\label{sec:proof-theor-refth}
\subsection{Definitions}
\label{sec:definitions}
First, we review the definition of a block optimal point and describe
some useful mathematical objects used on the proofs.  

\begin{definition}
\label{def:block-optimal}
A point $z^{\bullet} = (z^{\bullet}_{i})_{i \in \mathcal{V}}$ is
\emph{block optimal} for the function $F$
in~\eqref{eq:matrix-cost} if, for all $i$, $z^{\bullet}_{i} \in
\argmin_{w_{i} \in \mathcal{Z}_{i}} F(z^{\bullet}_{1}, \cdots, w_{i}, \cdots,
z^{\bullet}_{n})$~\cite{JakoveticXavierMoura2011}.
\end{definition}

We define the sets
\begin{eqnarray}
  \label{eq:Z-star}
  \mathcal{Z}^{\star} &=& \{z \in \mathcal{Z}: F(z) = F^{\star}\}\\
  \label{eq:X-epsilon}
  \mathcal{Z}_{\epsilon} &=& \left \{z : \mathrm{d}_{\mathcal{Z}^{\star}}(z) < \epsilon\right \}\\
  \label{eq:x-epsilon-c}
  \mathcal{Z}_{\epsilon}^{c} &=& \left \{z : \mathrm{d}_{\mathcal{Z}^{\star}}(z) \geq \epsilon\right \}\\
  \label{eq:sublevel-set}
  \mathcal{Z}_{F} &=& \left \{ z : F(z) \leq F\left(z^{0}\right) \right\}\\
  \label{eq:X-epsilon-c-f}
  \mathcal{\hat Z}_{\epsilon}^{c} &=& \mathcal{Z}_{F} \cap \mathcal{Z}_{\epsilon}^{c},
\end{eqnarray}
where $\mathcal{\hat Z}_{\epsilon}^{c}$ is the set of all points whose distance
to the optimal set $\mathcal{Z}^{\star}$ is larger than $\epsilon$, but also
belong to the sublevel set of $F$. We will see that the iterates
of Algorithm~\ref{alg:asyncronous} will belong to~$\mathcal{\hat Z}_{\epsilon}^{c}$
until they reach the absorbing set~$\mathcal{Z}_{\epsilon}$.
We also define the \emph{expected improvement function} as
\begin{equation}
  \label{eq:rel-improvement}
  \psi(z) = \expect \left [ F\left(Z(k+1) \right) | Z(k) = z \right] - F(z),
\end{equation}
and the coordinate optimal function as
\begin{equation}
  \label{eq:coord-opt-fn}
  F^{i}(z) = \min_{w_{i} \in \mathcal{Z}_{i}} F (z_{1}, \cdots, w_{i},\cdots, z_{n}).
\end{equation}
It is easy to see that the expected improvement $\psi$ can also be
written as
\begin{equation}
  \label{eq:rel-improvement-2}
  \psi(z) = \sum_{i=1}^{n} \left( F^{i}(z) - F(z) \right) P_{i},
\end{equation}
where $P_{i}$ is the probability of the event ``node~$i$ is awaken at
time~$t$'' (we recall the independence of the random
variables~$\chi_{t}$ defined in~\eqref{eq:rv}). For notational
convenience, we introduce the function
\begin{equation}
  \label{eq:phi}
  \phi(z) = - \psi(z),
\end{equation}
which, by construction of Algorithm~\ref{alg:asyncronous}, is always
non-negative.

\subsubsection{Auxiliary Lemmas}
\label{sec:auxil-lemmas}

The analysis is founded in
Lemma~\ref{lem:expected-improvement-properties}, where the symmetric
of the expected improvement is said to attain a positive infimum on
the set~$\mathcal{\hat Z}_{\epsilon}^{c}$. Lemma~\ref{lem:basic-properties} will
be instrumental in the proof of
Lemma~\ref{lem:expected-improvement-properties} but contains also some
useful properties of function~$F$ and the solution
set~$\mathcal{Z}^{\star}$.

\begin{lemma}[Basic properties]
  \label{lem:basic-properties}
  Let $F$ as defined in~\eqref{eq:matrix-cost}. Then the following
  properties hold.
  \begin{enumerate}
  \item $F$ is coercive;
  \item $F^{\star} \geq 0$ and $\mathcal{Z}^{\star} \neq
    \varnothing$;
  \item $\mathcal{Z}^{\star}$ is compact;
  \item If $z^{\bullet}$ is block optimal for $F$ in $\mathcal{Z}$, then it is
    global optimal for $F$ in $\mathcal{Z}$.
  \end{enumerate}
\end{lemma}
\begin{proof}
  \begin{enumerate}[leftmargin=*,noitemsep,topsep=0pt,parsep=0pt,partopsep=0pt]
  \item By Assumption~\ref{th:connected-assumption} there is a path
    from each node~$i$ to some node~$j$ which is connected to an
    anchor~$k$. Also, we know that, by definition,~$\|y_{ij}\|$
    and~$\|w_{ik}\|$ are bounded by the ranges~$d_{ij} < \infty$
    and~$r_{ik} < \infty$. So these components of~$z$ will have no
    effect in the limiting behavior of~$F$. If $\|x_{i}\| \to \infty$
    there are two cases: (1) there is at least one edge $t \sim u$
    along the path from~$i$ to~$j$ where~$\|x_{t}\| \to \infty$
    and~$\|x_{u}\|\not \to \infty$, and so
    $h_{R_{tu}}(\|x_{t}-x_{u} - y_{tu}\|) \to \infty$; (2) if
    $\|x_{u}\| \to \infty$ for all~$u$ in the path between~$i$
    and~$j$, in particular we have~$\|x_{j}\| \to \infty$ and
    so~$h_{Ra_{jk}}(\|x_{j}-a_{k} - w_{jk}\|) \to \infty$, and
    in both cases~$F \to \infty$, thus,~$F$ is coercive.
  \item Function~$F$ defined in~\eqref{eq:matrix-cost} is a
    continuous, convex and real valued function lower bounded by zero;
    so, the infimum~$F^{\star}$ exists and is non-negative. To prove
    this infimum is attained
    and~$\mathcal{Z}^{\star} \neq \varnothing$, we observe that the
    set~$\mathcal{Z}$ is a cartesian product of the closed
    sets~$\reals^{np}$, $\mathcal{Y}$ and~$\mathcal{W}$, and
    so~$\mathcal{Z}$ is also closed. Now consider the
    set~$T_{\alpha} = \{z : F(z) \leq \alpha \}$; $T_{\alpha}$ is a
    sublevel set of a continuous, coercive function and, thus, it is
    compact. For some~$\alpha$, the intersection of~$T_{\alpha}$
    and~$\mathcal{Z}$ is nonempty and it is known that the
    intersection of a closed and a compact set is
    compact~\cite[Corollary to 2.35]{rudin1976principles},
    so~$T_{\alpha}\cap\mathcal{Z}$ is compact. As function~$F$ is
    convex, it is also continuous on the compact
    set~$T_{\alpha}\cap\mathcal{Z}$, and by the extreme value theorem,
    the value~$p = \inf_{z \in T_{\alpha}\cap\mathcal{Z}} F(z)$ is
    attained and it is obvious
    that~$\inf_{z \in T_{\alpha}\cap\mathcal{Z}} F(z) = \inf_{z \in
      \mathcal{Z}} F(z)$.
  \item $\mathcal{Z}^{\star}
    = T_{\alpha} \cap \mathcal{Z}$ for~$\alpha =
    F^{\star}$, and we deduced in the previous proof that~$ T_{\alpha}
    \cap \mathcal{Z}$ is compact.
  \item If~$z^{\bullet}$
    is block-optimal, then~$\langle
    \nabla F_i(z^\bullet_i), z_i-z_i^\bullet\rangle \geq
    0$ for all~$z_i \in \mathcal{Z}_i$ and for
    all~$i$.
    When stacking the inequalities for all~$i$,
    we get~$\langle
    \nabla{F}(z^\bullet),z-z^\bullet\rangle \geq
    0$, which proves the claim.\qedhere
  \end{enumerate}
\end{proof}

\begin{lemma}
\label{lem:expected-improvement-properties}
  Let~$\phi$ be defined as~\eqref{eq:phi}, taking values on the
  set~$\mathcal{\hat Z}_{\epsilon}^{c}$ in~\eqref{eq:X-epsilon-c-f}. Then,
  \begin{enumerate}
  \item Function $\phi$ is positive:
    \begin{equation}
      \label{eq:expected-improvement-property-1}
      \phi(z) > 0, \quad \text{for all }  z \in \mathcal{\hat Z}_{\epsilon}^{c};
    \end{equation}
  \item And, as a consequence, function~$\phi$ is bounded by a finite
    positive value~$a_{\epsilon}$:
    \begin{equation}
      \label{eq:expected-improvement-property-2}
      \inf_{z \in \mathcal{\hat Z}_{\epsilon}^{c}} \phi(z) = a_{\epsilon}.
    \end{equation}
  \end{enumerate}
\end{lemma}
\begin{proof}
  We start by proving the first claim, $\phi(z) > 0$ for all $z \in
  \mathcal{\hat Z}_{\epsilon}^{c}$. Suppose $\phi(z) = 0$; then, by
  Equation~\eqref{eq:rel-improvement-2}
  \begin{equation*}
    F^{i}(z) = F(z),
  \end{equation*}
  which means $z$ is block optimal; by
  Lemma~\ref{lem:basic-properties}, $z$ is, then, global optimal,
  which contradicts the fact that $z$ belongs to the set~$
  \mathcal{\hat Z}_{\epsilon}^{c}$.  The second claim follows by observing that
  $\phi$ is a sum of real valued functions and, thus, a real valued
  function, and that $\phi(z)$ is bounded below by zero in~$
  \mathcal{\hat Z}_{\epsilon}^{c}$ and so it has a positive infimum for~$
  \mathcal{\hat Z}_{\epsilon}^{c}$.
\end{proof}

\subsubsection{Theorems}
\label{sec:theorems}
Equipped with the previous Lemmas, we are now ready to prove the
Theorems stated in Section~\ref{sec:analysis-asynchr-algor}.
\begin{proof}[Proof of Theorem~\ref{th:asyn-convergence}]
  We denote the random variable corresponding to the outcome of the
  $t$-th loop step of Algorithm~\ref{alg:asyncronous} as~$Z^{t}$. The expected
  value of the expected improvement function~$\psi$ is
  \begin{IEEEeqnarray*}{rCl}
    \expect \left[ \psi\left(Z^{t}\right) \right] &=& 
        \expect \left[ \expect \left[F \left(Z^{t+1} \right) | 
            Z^{t} \right] \right] - 
        \expect\left[ F\left(Z^{t}\right) \right]\\
        &=& \expect \left [F \left( Z^{t+1} \right) \right] -
        \expect \left[F \left( Z^{t} \right) \right],
  \end{IEEEeqnarray*}
  where the second equality comes from the tower property documented,
  \eg, in Williams~\cite{Williams1991}. This expectation can also be
  written as
\begin{IEEEeqnarray*}{rCl}
  \expect \left[ \psi\left(Z^{t}\right) \right] &=&  \expect \left[ \psi \left(Z^{t} \right) | Z^{t} \in \mathcal{\hat Z}_{\epsilon}^{c} \right] \prob \left(Z^{t} \in \mathcal{\hat Z}_{\epsilon}^{c} \right) \\
&&+ \expect \left[ \psi \left(Z^{t} \right) | Z^{t} \not \in \mathcal{\hat Z}_{\epsilon}^{c} \right] \prob \left(Z^{t} \not \in \mathcal{\hat Z}_{\epsilon}^{c} \right)\\
&\leq& \expect \left[ \psi \left(Z^{t} \right) | Z^{t} \in \mathcal{\hat Z}_{\epsilon}^{c} \right] \prob \left(Z^{t} \in \mathcal{\hat Z}_{\epsilon}^{c} \right).
\end{IEEEeqnarray*}
By combining both we get
\begin{IEEEeqnarray*}{l}
  \expect \left [F \left( Z^{t+1} \right) \right] - \expect
  \left[F \left( Z^{t} \right) \right] \\
  \leq \expect \left[ \psi
    \left(Z^{t} \right) | Z^{t} \in \mathcal{\hat Z}_{\epsilon}^{c} \right] \prob
  \left(Z^{t} \in \mathcal{\hat Z}_{\epsilon}^{c} \right)
\end{IEEEeqnarray*}
which can be further bounded using
Lemma~\ref{lem:expected-improvement-properties} as
\begin{equation*}
  \expect \left [F \left( Z^{t+1} \right) \right] - \expect
  \left[F \left( Z^{t} \right) \right] \leq -a_{\epsilon} p_{t}
\end{equation*}
where $p_{t} = \prob \left(Z^{t} \in \mathcal{\hat Z}_{\epsilon}^{c} \right)$. By
expanding the recursion we obtain
\begin{equation*}
  \expect \left [F \left( Z^{t+1} \right) \right] \leq
  -a_{\epsilon}\sum_{k=1}^{t}p_{k} + F\left(z^{0}\right),
\end{equation*}
which provides a bound on the sum of probabilities~$p_{k}$ when
rearranged as
\begin{equation*}
  \sum_{k=1}^{t}p_{k} \leq \frac{F \left( z^{0} \right) - \expect \left [F \left( Z^{t+1} \right) \right]}{a_{\epsilon}}.
\end{equation*}
Taking $t$ up to infinity, we obtain
\begin{IEEEeqnarray*}{rCl}
  \sum_{k=1}^\infty p_k &\leq& \frac{F \left( z^{0} \right) - \expect \left [F \left( Z(\infty) \right) \right]}{a_\epsilon} \\
&\leq& \frac{F \left( z^{0} \right) - \hat{f}^\star}{a_\epsilon}.
\end{IEEEeqnarray*}
This means the infinite series of probabilities $p_{k}$ assumes a
finite value; by the Borel-Cantelli Lemma, we get
\begin{equation*}
  \prob \left( Z^{t} \in \mathcal{\hat Z}_{\epsilon}^{c}, \quad i.o. \right) = 0,
\end{equation*}
where~$i.o.$ stands for \emph{infinitely often}. This concludes the
proof, since this statement is equivalent to the first claim of
Theorem~\ref{th:asyn-convergence}.
\end{proof}

\begin{proof}[Proof of Theorem~\ref{th:nr-iter-convergence}]
  Consider redefining the sets in~\eqref{eq:X-epsilon}
  and~\eqref{eq:x-epsilon-c} as
  \begin{IEEEeqnarray*}{rCl}
    \mathcal{Y}_\epsilon &=& \left\{ z : F(z) < F^\star +
      \epsilon \right \}\\
    \mathcal{Y}_\epsilon^c &=& \left\{ z : F(z) \geq F^\star +
      \epsilon \right \},
  \end{IEEEeqnarray*}
  thus leading to
  \begin{equation*}
    \mathcal{\hat{Y}}_\epsilon^c = \mathcal{Y}_\epsilon^c \cap
    \mathcal{Z}_{\hat{f}}.
  \end{equation*}
Using the same arguments as in
Lemma~\ref{lem:expected-improvement-properties}, we can prove that
\begin{equation*}
  \inf_{z \in \hat Y_{\epsilon}^{c}} \phi (z) = b_{\epsilon} < \infty.
\end{equation*}
We now define a sequence of points $\tilde z ^{t}$ such that
\begin{equation*}
  \tilde z^{t} = 
  \begin{cases}
    z^{t} & \text{if } z^{t} \in \mathcal{\hat Y}_{\epsilon}^{c} \\
    z^{\star} & \text{otherwise}
  \end{cases},
\end{equation*}
and the sequence of real values
\begin{equation*}
  \psi(\tilde z^{t}) =
  \begin{cases}
    \psi(z^{t})  & \text{if } z^{t} \in \mathcal{\hat Y}_{\epsilon}^{c} \\
    0  & \text{otherwise}
  \end{cases}.
\end{equation*}
The expected value of $\psi(\tilde Z^{t})$ is
\begin{equation*}
  \expect \left[\psi \left(\tilde Z^{t} \right) \right] = \expect
  \left[ F \left(\tilde Z^{t+1} \right) \right] - \expect \left[
    F \left(\tilde Z^{t} \right) \right]. 
\end{equation*}
Summing these expectations over time, we get
\begin{equation*}
  \sum_{k=0}^{t-1} \expect \left[\psi \left(\tilde Z^{t} \right)
  \right] = \expect \left[ F \left( \tilde Z(t) \right) \right] -
  F \left( z^{0} \right). 
\end{equation*}
Taking $t$ to infinity and interchanging integration and summation we obtain
\begin{equation*}
  \expect \left[ \sum_{k=0}^{\infty} \psi\left( \tilde Z^{t} \right)
  \right] = \expect \left[ F \left( Z(\infty) \right) \right] -
  F\left( z^{0} \right).
\end{equation*}
From the definition of~$\psi(\tilde z^{t})$ we can write
\begin{equation*}
  \expect \left[ \sum_{k=0}^{\infty} \psi\left( \tilde Z^{t} \right)
  \right] \leq \expect \left[ K_{\epsilon}(-b_{\epsilon}) \right],
\end{equation*}
thus obtaining the result
\begin{IEEEeqnarray*}{rCl}
  \expect \left[ K_\epsilon \right] &\leq& \frac{F (z^{0}) -
    \expect \left[ F \left (Z(\infty)\right) \right]}{b_\epsilon}
  \\ 
  &\leq& \frac{F(z^{0}) - F^\star}{b_\epsilon}
\end{IEEEeqnarray*}
which is a finite number. This completes the proof.
\end{proof}

\bibliographystyle{IEEEtran}
\bibliography{biblos}

\end{document}